\documentclass[11pt]{amsart}

\usepackage{comment}
\usepackage{color}
\usepackage{amsmath,mathabx}
\usepackage{amsthm}
\usepackage{amssymb}
\usepackage{graphicx}
\usepackage{enumerate}
\usepackage{amsfonts}
\usepackage{palatino}
\usepackage{mathrsfs}
\usepackage{mathdots}
\usepackage{color}

\numberwithin{equation}{section}
\theoremstyle{plain}

\newtheorem{Lemma}{Lemma}[section]
\newtheorem{Proposition}[Lemma]{Proposition}
\newtheorem{Corollary}[Lemma]{Corollary}
\newtheorem*{Corollary*}{Corollary}
\newtheorem{Theorem}[Lemma]{Theorem}
\newtheorem*{Theorem*}{Theorem}
\theoremstyle{definition}

\newtheorem{Example}[Lemma]{Example}
\newtheorem{Remark}[Lemma]{Remark}

\allowdisplaybreaks

\usepackage{enumitem}
\setlist[enumerate,1]{label=(\alph*),font=\upshape}

\setlist[enumerate,2]{label=(\roman*),font=\upshape}

\newtheorem{question}[equation]{Question}

\def\HH{\mathscr{H}}
\def\MM{\mathscr{M}}

\def\P{\mathscr{P}}

\def\C{\mathbb{C}}

\def\D{\mathbb{D}}
\def\T{\mathbb{T}}

\def\phi{\varphi}

\renewcommand{\ker}{\operatorname{Ker}}

\newcommand{\beqa}{\begin{eqnarray*}}
\newcommand{\eeqa}{\end{eqnarray*}}

\renewcommand{\leq}{\leqslant}
\renewcommand{\le}{\leqslant}
\renewcommand{\subset}{\subseteq}

\title{Cyclicity in de Branges--Rovnyak spaces}

\author[Fricain]{Emmanuel Fricain}
 \address{Laboratoire Paul Painlev\'e, Universit\'e de Lille, 59 655 Villeneuve d'Ascq C\'edex }
 \email{emmanuel.fricain@univ-lille.fr}

\author[Grivaux]{Sophie Grivaux}
\address{CNRS, Laboratoire Paul Painlev\'e, Universit\'e de Lille, 59 655 Villeneuve d'Ascq C\'edex }
\email{sophie.grivaux@univ-lille.fr}

\thanks{The authors were supported by Labex CEMPI (ANR-11-LABX-0007-01) and the project FRONT (ANR-17-CE40 - 0021).}


\begin{document}

\begin{abstract}
In this paper, we study the cyclicity problem with respect to the forward shift operator $S_b$ acting on the de Branges--Rovnyak space $\HH(b)$ associated to a function $b$ in the closed unit ball of $H^\infty$ and satisfying $\log(1-|b|)\in L^1(\mathbb T)$. We present a characterisation of cyclic vectors for $S_b$ when $b$ is a rational function which is not a finite Blaschke product. This characterisation can be derived from the description,
given in  \cite{MR4246975},
of invariant subspaces of $S_b$ in this case, but we provide here an elementary proof. We also study the situation where $b$ has the form $b=(1+I)/2$, where $I$ is a non-constant inner function such that the associated model space $K_I=\HH(I)$ has an orthonormal basis of reproducing kernels. 

\end{abstract}

\maketitle
\par\bigskip
\hfill \emph{To the memory of Mohamed Zarrabi (1964 -- 2021)}
\par\bigskip
\section{Introduction}
In this paper, we study the cyclicity problem with respect to the forward shift operator $S_b$ acting on the de Branges--Rovnyak space $\HH(b)$, associated to a function $b$ belonging to the closed unit ball of $H^\infty$ and satisfying $\log(1-|b|)\in L^1(\mathbb T)$. This problem of cyclicity has a long and outstanding history and many efforts have been dedicated to solving it in various reproducing kernel Hilbert spaces. It finds its roots in the pioneering work of Beurling who showed that cyclicity of a function $f$ in the Hardy space $H^2$ is equivalent to $f$ being outer. Brown and Shields studied the cyclicity problem in the Dirichlet spaces $\mathcal{D}_\alpha$  for polynomials that do not have zeros inside the disc, but that may have some zeros on its boundary. Such functions are cyclic in $\mathcal{D}_\alpha$ if and only if $\alpha \leq 1$. They also proved that the set of zeros on the unit circle (in the radial sense) of cyclic functions in the Dirichlet space $\mathcal{D}_\alpha$  has zero logarithmic capacity, and this led them to ask whether any outer function with this property is cyclic \cite{BS1}. This problem is still open although there has been relevant contributions to the topic by a number of authors; e.g. see \cite{B2, BC1, EFKR1, EFKR2, HS1, RS1}. We also mention the paper \cite{MR3475457} where the authors prove the Brown--Shields conjecture in the context of some particular Dirichlet type spaces $\mathcal D(\mu)$, which happen to be related to our context of de Branges--Rovnyak spaces \cite{MR3110499}. 
\par\smallskip
The de Branges--Rovnyak spaces $\HH(b)$ (see the precise definition in Section~\ref{sec2}) have been introduced by L. de Branges and J. Rovnyak in the context of model theory (see \cite{MR0215065}). A whole class of Hilbert space contractions is unitarily equivalent to $S^*|\HH(b)$, for an appropriate function $b$ belonging to the closed unit ball of $H^{\infty}$. Here $S$ is the forward shift operator on $H^2$ and $S^*$, its adjoint, is the backward shift operator on $H^2$. The space $\HH(b)$ is invariant with respect to $S^*$ for every $b$ in the closed unit ball of $H^\infty$, and $S^*$ defines a bounded operator on $\HH(b)$, endowed with its own Hilbert space topology.
On the contrary, $\HH(b)$ is invariant with respect to $S$ if and only if $\log(1-|b|)\in L^1(\mathbb T)$, i.e. if and only if $b$ is a non-extreme point of the closed unit ball of $H^{\infty}$. See \cite[Corollary 25.2]{FM2}. In \cite{MR3309352}, it is proved that if $\log(1-|b|)\in L^1(\mathbb T)$, the cyclic vectors of $S^*|\HH(b)$ are precisely the cyclic vectors of $S^*$ which live in $\HH(b)$. Note that cyclic vectors of $S^*$ have been characterised by Douglas--Shapiro--Shields \cite{MR203465}. The result of \cite{MR3309352} is based on a nice description, due to Sarason, of closed invariant subspaces of $S^*|\HH(b)$. See also \cite[Corollary 24.32]{FM2}. Unfortunately an analogous description of closed invariant subspaces of $S_b=S|\HH(b)$ remains an unsolved and difficult problem. In \cite{MR4246975}, Gu--Luo--Richter give an answer in the case where $b$ is a rational function which is not inner, generalising a result of Sarason \cite{MR847333}.
\par\smallskip
The purpose of this paper is to study the cyclic vectors of $S_b$ when the function $b$ is such that $\log(1-|b|)\in L^1(\mathbb T)$. In Section~\ref{sec2}, we present a quick overview of some useful properties of de Branges--Rovnyak spaces. Then, in Section~\ref{sec3}, we give some general facts on cyclic vectors and completely characterise holomorphic functions in a neighborhood of the closed unit disc which are cyclic for $S_b$. In Section~\ref{sec4}, we give a characterisation of cyclic vectors for $S_b$ when $b$ is rational (and not inner). Of course, this characterisation can be derived from the description,
given in  \cite{MR4246975},
of invariant subspaces of $S_b$ when $b$ is a non-inner rational function. Nevertheless, we will give  a more direct and easier proof of this characterisation. Finally, Section~\ref{sec5} will be devoted to the situation where $b=(1+I)/2$, where $I$ is a non-constant inner function such that the associated model space $K_I=\HH(I)$ has an orthonormal basis of reproducing kernels.

\par\bigskip
 \section{Preliminaries on $\HH(b)$ spaces}\label{sec2}
\subsection{Definition of de Branges-Rovnyak spaces} Let  
 $$\operatorname{ball}(H^{\infty}) := \Big\{b \in H^{\infty}: \|b\|_{\infty} = \sup_{z \in \D} |b(z)| \leq 1\Big\}$$ be the closed unit ball of $H^{\infty}$, the space of bounded analytic functions on the open unit disk $\D$, endowed with the sup norm.
For $b \in \operatorname{ball}(H^{\infty})$, the \emph{de Branges--Rovnyak space} $\HH(b)$ is the reproducing kernel Hilbert space on $\mathbb D$ associated with the positive definite kernel  $k_{\lambda}^b$, $\lambda\in\D$, defined as
\begin{equation}\label{eq:original kernel H(b)}
k^{b}_{\lambda}(z) = \frac{1 -  \overline{b(\lambda)}b(z)}{1 - \overline{\lambda} z}, \quad  z \in \D.
\end{equation}
 
It is known that $\HH(b)$ is contractively contained in the well-studied Hardy space $H^2$ of analytic functions $f$ on $ \D $ for which 
$$\|f\|_{H^2} := \Big(\sup_{0 < r < 1} \int_{\T} |f(r \xi)|^2 dm(\xi)\Big)^{\frac{1}{2}}<\infty,$$
where $m$ is the normalised Lebesgue measure on the unit circle $\T = \{\xi \in \C: |\xi| = 1\}$ \cite{Duren, garnett}. For every $f \in H^2$, the radial limit 
$\lim_{r \to 1^{-}} f(r \xi) =: f(\xi)$ (even the non-tangential limit $f(\xi):=\lim_{\substack{z\to\xi\\  \sphericalangle}}f(z)$) exists for $m$-a.e. $\xi \in \T$, and
\begin{equation}\label{knnhHarsysu}
\|f\|_{H^2}  = \Big(\int_{\T} |f(\xi)|^2 dm(\xi)\Big)^{\frac{1}{2}}.
\end{equation}
Though $\HH(b)$ is contractively contained in $H^2$, it is generally not closed in the $H^2$ norm. It is known that $\HH(b)$ is closed in $H^2$ if and only if $b=I$ is an inner function, meaning that $|I(\zeta)|=1$ for a.e. $\zeta\in\mathbb T$. In this case, $\HH(b)=K_I=(I H^2)^\perp$ is the so-called \emph{model space} associated to $I$. Note that $K_I=H^2\cap I\overline{zH^2}$ (see \cite[Proposition 5.4]{MR3526203}), and then $K_I=\ker T_{\bar{I}}$, where $T_{\bar{I}}$ is the Toeplitz operator with symbol $\bar I$ defined on $H^2$ as $T_{\bar I}f=P_+(\bar I f)$, where $P_+$ denotes the orthogonal projection from $L^2$ onto $H^2$.  

\par\smallskip
We refer the reader to the book \cite{MR1289670} by Sarason and to the monograph \cite{FM1}, \cite{FM2} by Fricain and Mashreghi for an in-depth study of de Branges-Rovnyak spaces and their connections to numerous other topics in operator theory and complex analysis.
\par\smallskip
% \subsection{The pythagorean mate of $b$}
In this paper, we will always assume that $b$ is a non-extreme point of $\operatorname{ball}(H^{\infty})$, which is equivalent to requiring that $\log(1-|b|)\in L^1(\mathbb T)$. Under this assumption, there is a unique outer function $a$, called \emph{the pythagorean mate} for $b$, such that $a(0)>0$ and $|a|^2+|b|^2=1$ a.e. on $\mathbb T$. There are two important subspaces of $\HH(b)$ which can be defined via this function $a$. The first one is the space $\MM(a)=aH^2$, equipped with the range norm
\[
\|af\|_{\MM(a)}=\|f\|_2,\qquad f\in H^2.
\]
The second one is $\MM(\bar a)=T_{\bar a}H^2$, equipped also with the range norm
\[
\|T_{\bar a}f\|_{\MM(\bar a)}=\|f\|_2,\qquad f\in H^2.
\]
Note that since $a$ is outer, the Toeplitz operator $T_{\bar a}$ is one-to-one and so the above norm is well defined. It is known that $\MM(a)$ is contractively contained into $\MM(\bar a)$, which itself is contractively contained into $\HH(b)$. See \cite[Theorem 23.2]{FM2}. Note that $\MM(a)$ is not necessarily closed in $\HH(b)$. See \cite[Theorem 28.35]{FM2} for a characterisation of closeness of $\MM(a)$ in $\HH(b)$-norm. There is also an important relation between $\HH(b)$ and $\MM(\bar a)$ which gives a recipe to compute the norm in $\HH(b)$. Indeed, if $f\in H^2$, then $f\in\HH(b)$ if and only if there is a function $f^+\in H^2$ satisfying $T_{\bar b}f=T_{\bar a}f^+$ (and then $f^+$ is necessarily unique). Moreover, in this case, we have 
\[
\|f\|_b^2=\|f\|_2^2+\|f^+\|_2^2.
\]
Also, 
\begin{equation}\label{Inconnu}
 \langle f,g \rangle_b=\langle f,g \rangle_2+\langle f^+,g^+ \rangle_2\quad\textrm{ for every }f,g\in\HH(b).
\end{equation}
See \cite[Theorem 23.8]{FM2}. Finally, let us recall that $\HH(b)=\MM(\bar a)$ if and only if $(a,b)$ forms a corona pair, that is 
\[
\tag{HCR} \qquad \inf_{\mathbb D}(|a|+|b|)>0.
\]
See \cite[Theorem 28.7]{FM2}.

\par\smallskip
A crucial fact on de Branges-Rovnyak space is that the space $\HH(b)$ is invariant with respect to the shift operator $S: f \mapsto zf$ if and only if the function $b$ is non-extreme. Since we will consider in this paper only the case where $b$ is non-extreme, $\HH(b)$ is  indeed invariant by $S$, and $S$ defines a bounded operator on $\HH(b)$, endowed with its own Hilbert space topology, which we will denote by $S_b$. The functions $z^n$ belong to $\HH(b)$
 for every $n\ge 0$. Actually, we have
\begin{equation}\label{eq:density-polynomial}
\mbox{Span}(z^n:n\geq 0)=\HH(b),
\end{equation}
where $\mbox{Span}(A)$ denotes the closed linear span generated by vectors from a certain family $A$.
In other words, the polynomials are dense in $\HH(b)$. See \cite[Theorem 23.13]{FM2}.
Note that \eqref{eq:density-polynomial} exactly means that the constant function $1$ is cyclic for $S_b$. 
\par\smallskip
Another tool which will turn out to be useful when studying the cyclicity for the shift operator is the notion of \emph{multiplier}. Recall that the set $\mathfrak M(\HH(b))$ of multipliers of $\HH(b)$ is defined as
\[
\mathfrak M(\HH(b))=\{\varphi\in\mbox{Hol}(\mathbb D): \varphi f\in \HH(b),\forall f\in\HH(b)\}.
\]
Using the closed graph theorem,  it is easy to see that when $\varphi\in \mathfrak M(\HH(b))$, then $M_\varphi$, the multiplication operator by $\varphi$, is bounded on $\HH(b)$. The algebra of multipliers is  a Banach algebra when equipped with the norm $\|\varphi\|_{\mathfrak M(\HH(b)}=\|M_\varphi\|_{\mathcal L(\HH(b))}$.  Using standard arguments, we see that $\mathfrak M(\HH(b))\subset H^\infty\cap\HH(b)$. In general, this inclusion is strict. See \cite[Example 28.24]{FM2}. However, we will encounter below a situation (when $b$ is a rational function which is not a finite Blaschke product) where we have the equality $\mathfrak M(\HH(b))= H^\infty\cap\HH(b)$. 

\par\smallskip

\subsection{Some properties of the reproducing kernels of $H^2$ in de Branges-Rovnyak spaces}
Recall that we are supposing that $b$ is non-extreme. If we denote by $k_\lambda(z)=(1-\overline{\lambda}z)^{-1}$ the reproducing kernel of $H^2$ at the point $\lambda\in\D$, then $k_\lambda$ belongs to $\HH(b)$ and 
\begin{equation}\label{eq:density-crk}
\mbox{Span}(k_\lambda:\lambda\in\mathbb D)=\HH(b).
\end{equation}
See \cite[Corollary 23.26]{FM2} or \cite[Lemma 7]{MR3390195}. We also know (see \cite[Theorem 23.23]{FM2}) that $bk_\lambda\in \HH(b)$ for every $\lambda\in\mathbb D$, and that  for every $f\in\HH(b)$ we have
\begin{equation}\label{eq1EZD:lem-completeness}
\langle f,k_\lambda \rangle_b=f(\lambda)+\frac{b(\lambda)}{a(\lambda)}f^+(\lambda)\quad\mbox{and}\quad \langle f,bk_\lambda\rangle_b=\frac{f^+(\lambda)}{a(\lambda)}\cdot
\end{equation}
Using these two equations, we can produce an interesting complete family in $\HH(b)$ which will be of use to us.

\begin{Lemma}\label{Lem:completenes}
Let $b$ be a non-extreme point of the closed unit ball of $H^{\infty}$, and let
 $c$ be a complex number with $|c|<1$. Then 
\[
\mbox{Span}(k_\mu-cbk_\mu:\mu\in\mathbb D)=\HH(b).
\]
\end{Lemma} 

\begin{proof}
Let $h\in\HH(b)$, and assume that for every $\mu\in\mathbb D$, $h$ is orthogonal in $\HH(b)$ to  $k_\mu-cbk_\mu$. According to \eqref{eq1EZD:lem-completeness}, we have
\[
0=h(\mu)+\frac{b(\mu)}{a(\mu)}h^+(\mu)-\overline{c}\frac{h^+(\mu)}{a(\mu)}\cdot
\]
This can be rewritten as $ah=-bh^++\overline{c}h^+$. Multiplying this equality by $\bar b$ and using the fact that $|a|^2+|b|^2=1$ a.e. on $\mathbb T$, we obtain
\[
a(\bar b h-\bar a h^+)=-(1-\overline{c}\bar b)h^+.
\]
Note that $|1-\overline{c}\bar b|\geq 1-|c|>0$, and so the last identity can be written as 
\[
\frac{\bar bh-\bar ah^+}{1-\overline{c}\bar b}=-\frac{h^+}{a}\cdot
\]
On the one hand, this equality says that $\frac{h^+}{a}$ belongs to $ L^2$ and since $a$ is outer, we have $\frac{h^+}{a}\in H^2$. See \cite[page 43]{MR1864396}. On the other hand, by definition of $h^+$, the function $\bar bh-\bar ah^+$ belongs to $\overline{H^2_0}$ and since $(1-\overline{c}\bar b)^{-1}$ is in $ \overline{H^\infty}$, we also have $\frac{h^+}{a}\in \overline{H^2_0}$. Then $\frac{h^+}{a}$ belongs to $ H^2\cap \overline{H^2_0}=\{0\}$. Finally we get that $h^+=0$ and thus that $h=0$.
\end{proof}

\par\smallskip

\subsection{Boundary evaluation points on $\HH(b)$}
An important tool in the cyclicity problem will be the boundary evaluation points for $\HH(b)$. It is known that the description of these points depends on the inner-outer factorisation of $b$. Recall that any $b$ in $\operatorname{ball}(H^{\infty})$ can be decomposed as
\begin{equation} \label{E:dec-b-bso}
b(z)= B(z) S_\sigma(z) O(z), \qquad z \in \D,
\end{equation}
where
\[
B(z) = \gamma\,\prod_{n\ge 1} \left(\, \frac{|a_n|}{a_n}\frac
{a_n-z}{1-\overline{a}_nz} \,\right) \quad\textrm{is a Blaschke product,} 
\]
with  $|\gamma|=1$, $a_n\in\D$ for every $n\ge1$,  and $\sum_{n\ge 1}(1-|a_n|)<+\infty$,
\[
S_\sigma(z) = \exp\left( -\int_\T \frac{\xi+z}{\xi-z} \, d\sigma(\xi) \right) \quad\textrm{is a singular inner function,}
\]
with $\sigma$ a positive finite Borel measure on $\T$ which is singular with respect to the Lebesgue measure, and 
\[
O(z) = \exp\left( \int_\T \frac{\xi+z}{\xi-z} \log|b(\xi)| \,dm(\xi) \right)
\]
is the outer part of $b$.
Now, let $E_0(b)$ be the set of points $\zeta\in\mathbb T$ satisfying the following condition:

 \begin{equation}
  \sum_n\frac{1-|a_n|^2}{|\zeta-a_n|^{2}}+\int_{\T} \frac{d\sigma(\xi)}{|\zeta-\xi|^{2}}+\int_{\T} \frac{\big| \log|b(\xi)| \big|}{|\zeta-\xi|^{2}} \,\, dm(\xi)<\infty.
 \end{equation}
 
It is proved in \cite{MR2390675} that for every point $\zeta\in\mathbb T$, every function $f\in \HH(b)$ has a non-tangential limit at $\zeta$ if and only if $\zeta\in E_0(b)$.  This is also equivalent to the property that $b$ has an {\emph{angular derivative (in the sense of Carath\'eodory) at $\zeta$}, meaning that $b$ and $b'$ both have a non-tangential limit at $\zeta$ and $|b(\zeta)|=1$. 
Moreover, in this case, the linear map 
\begin{equation}\label{eq:radial-limit-ADC}
f\longmapsto f(\zeta):=\lim_{\substack{z\to\zeta\\  \sphericalangle}}f(z)
\end{equation} 
is bounded on $\HH(b)$. 
The function $k_\zeta^b$ defined by
\[
 k_{\zeta}^b(z)=\frac{1 -  \overline{b(\zeta)}b(z)}{1 - \overline{\zeta} z}, \quad  z \in \D,
\]
belongs to $\HH(b)$, and
\[
 \langle f,k_{\zeta}^b\rangle_b=f(\zeta)\quad\textrm{ for every } f\in\HH(b).
\]
We call the function $k_\zeta^b$ \emph{the reproducing kernel of $\HH(b)$ at the point $\zeta$}, and \eqref{eq:radial-limit-ADC} means that the reproducing kernels $k_{z}^b$ tend weakly to $k_{\zeta}^b$ as $z\in\D$ tends non-tangentially to $\zeta$. See \cite[Theorem 25.1]{FM2}. There is also a nice connection between the boundary evaluation points and the point spectrum of $S_b^*$ in the case where $b$ is a non-extreme point in $\operatorname{ball}(H^{\infty})$: for $\zeta\in\mathbb T$, we have that
\begin{equation}\label{point-spectrum-boundary}
\bar\zeta \mbox{ is an eigenvalue for }S_b^* \mbox{ if and only if }  b \mbox{ has an angular derivative at }\zeta.
\end{equation}
\par\smallskip

The boundary evaluation points play a particular role in the description of certain orthogonal basis of  reproducing kernels in model spaces $K_I$, the 
so-called {\emph{Clark basis}}. Given an inner function $I$ and $\alpha\in\T$, recall that by Herglotz theorem, there is a unique finite positive Borel measure $\sigma_{\alpha}$ on $\mathbb T$, singular with respect to the Lebesgue measure, such that 
\begin{equation}\label{def-clark-measure}
\frac{1-|I(z)|^2}{|\alpha-I(z)|^2}=\int_{\mathbb T}\frac{1-|z|^2}{|\xi-z|^2}\,d\sigma_{\alpha}(\xi),\qquad z\in\mathbb D.
\end{equation}
% If $I(0)=0$, $\sigma_1$ is a probability measure.
The collection $(\sigma_{\alpha})_{\alpha\in\T}$ is the family of \emph{Clark measures} of $I$. 
\par\smallskip
Let $E_\alpha=\{\zeta\in E_0(I): I(\zeta)=\alpha\}.$ By \cite[Theorem 21.5]{FM2}, the point $\zeta$ belongs to $E_\alpha$ if and only if the measure $\sigma_\alpha$ has an atom at $\zeta$. In this case,
\begin{equation}\label{eq:nice-formula-derivative-norme-kernel-clark-measure}
\sigma_\alpha(\{\zeta\})=\dfrac{1}{|I'(\zeta)|}=\dfrac{1}{\|k_{\zeta}^I\|_2^2}\cdot
 \end{equation}
 See \cite[Theorems 21.1 and 21.5]{FM2}. 
When $\sigma_\alpha$ is a discrete measure, its support is exactly the set $E_\alpha$, which is necessarily countable, and we write it as 
\begin{equation}
E_\alpha=\{\zeta_n:n\geq 1\}=\{\zeta\in E_0(I): I(\zeta)=\alpha\}.
\end{equation}
Then, in this case, Clark  proved in \cite{MR301534} that the family $\{k_{\zeta_n}^I: n\geq 1\}$ forms an orthogonal basis of $K_I$ (and the family $\{{k_{\zeta_n}^I}/{\|k_{\zeta_n}^I\|_2}: n\geq 1\}$ forms an orthonormal basis of $K_I$). It is called \emph{the Clark basis of $K_I$} associated to point $\alpha\in\mathbb T$.

\subsection{A description of $\HH(b)$ when $b$ is a rational function}
Although the contents of the space $\HH(b)$ may seem mysterious for a general non-extreme $b \in \operatorname{ball}(H^{\infty})$, it turns out that when $b$ is a rational function (and not a finite Blaschke product -- in which case $b$ is an inner function, and thus extreme),  the description of $\HH(b)$ is quite explicit. Since our $b$ is a non-extreme point of $\operatorname{ball}(H^{\infty})$, it admits a pythagorean mate $a$, which is also a rational function. In fact, the function $a$ can be obtained from the Fej\'{e}r--Riesz theorem  (see \cite{MR3503356}). Let $ \zeta_1, \dots, \zeta_n $ denote the {\em distinct} roots of $ a $ on $ \T $, with corresponding  multiplicities $ m_1, \dots, m_n $, and define the polynomial $ a_1 $ by 
\begin{equation}\label{eq:definition of a}
a_1(z): =\prod_{k=1}^n (z-\zeta_k)^{m_k}.
\end{equation}
% When $a$ is the Pythagorean mate for $b$, let 
%\begin{equation}\label{nmnbhjhjk898}
%\widehat{a} = \prod_{j = 1}^{n} (z - \xi_j)^{m_j}.
%\end{equation}
% Note that $\MM(\overline{a}) = \MM(\overline{\widehat{a}})$ \cite[Prop.~2.7]{MR3967886}. 
%We start with a fixed rational, nonextreme function $ b $ in the unit ball of $ H^\infty $. The Pytagorean mate of $ b $ is the outer rational function $ a' $ with the property that $ |a'(\zeta)|^2+|b(\zeta)|^2=1 $ almost everywhere on $ \T $. We use this slightly unusual notation because we are interested not in $ a' $ but in the function 
Results from  \cite{MR3110499, MR3503356} show that $\HH(b)$ has an explicit description as 
\begin{equation}\label{eq:formula for H(b)}
\HH(b)=a_1H^2 \oplus \P_{N-1}=\mathcal{M}(a_1) \oplus \P_{N-1},
\end{equation}
where $ N=m_1+\dots+m_n $, and $\P_{N-1}$ denotes the set of polynomials of degree at most $N-1$. Since $a/a_1$ is invertible in $H^\infty=\mathfrak M(H^2)$, note that $\mathcal M(a)=\mathcal M(a_1)$. The notation $\oplus$ above denotes a topological direct sum in $\HH(b)$. But this sum may not be an orthogonal one. See \cite{MR3503356}. In particular, $\mathcal{M}(a_1) \cap \P_{N - 1} = \{0\}$. Moreover, if $ f\in\HH(b) $ is decomposed with respect to \eqref{eq:formula for H(b)} as 
\begin{equation}\label{uUUiipPPS}
 f=a_1\widetilde{f}+p, \quad  \mbox{where $\widetilde{f}  \in H^2$ and  $p \in \P_{N - 1}$},
 \end{equation}
an equivalent norm on $ \HH(b) $  (to the natural one induced by the positive definite kernel $k_{\lambda}^{b}$, $\lambda\in\D$, above)  is
\begin{equation}\label{eq:norm in h(b)}
\vvvert a_1\widetilde f+p\vvvert^{2}_{b}:=\|\widetilde{f}\|^2_{H^2}+\|p\|^2_{H^2}.
\end{equation}
Note that the functions $\widetilde{f}  \in H^2$ and  $p \in \P_{N - 1}$ appearing in the decomposition (\ref{uUUiipPPS}) are unique.
 It is important to note that $ \vvvert \cdot\vvvert_b $ is only equivalent to the original norm $\|\cdot\|_b$ associated to the kernel in~\eqref{eq:original kernel H(b)}, and its scalar product as well as the reproducing kernels and the adjoints of operators defined on $\HH(b)$ will be different. However, the cyclicity problem for $S_b$ does not depend on the equivalent norm we consider. So, in the rational case, there is no problem to work  with the norm given by \eqref{eq:norm in h(b)}.
 \par\smallskip
% %Since \eqref{eq:norm in h(b)} defines the norm and the scalar product on $ \HH(b) $ that we will use in the sequel, we prefer
% With the norm $\|\cdot\|_{b}$ and the corresponding inner product in mind, we need
%  to introduce a new notation for the associated reproducing kernels, different from~\eqref{eq:original kernel H(b)}, namely $ \kk^b_\lambda $ (note the bold face). By the term reproducing kernel we mean that $\kk^{b}_{\lambda} \in \HH(b)$ for all $\lambda \in \D$ and 
% \[
% \langle f, \kk^b_\lambda\rangle_b = f(\lambda) \quad \mbox{for all $ f\in\HH(b) $ and $ \lambda\in\D $.}
% \]
 \par\smallskip
 Note also that when the zeros $\zeta_1,\ldots,\zeta_n$ of the polynomial $a_1$ are simple (i.e. when $m_k=1$, $1\le k\le n$), then the space $\HH(b)$ coincides with a Dirichlet type space $\mathcal{D}(\mu)$, where $\mu$ is a finite sum of Dirac masses at the points $\zeta_k$, $1\le k\le n$. See \cite{MR3110499}. So our results are also connected to the works \cite{EFKR1} and \cite{EFKR2} on the cyclicity problem for Dirichlet spaces.
 
 \par\smallskip
 Using \eqref{uUUiipPPS} and the standard estimate that any $g \in H^2$ satisfies 
 \begin{equation}\label{gggbigggooooo}
 |g(z)| \leq \frac{\|g\|_{2}}{\sqrt{1 - |z|^2}} \quad \mbox{for all $z \in \D$},
 \end{equation}
 we see that for fixed $1 \leq k \leq n$ and for each $f \in \HH(b)$ we have 
 \begin{equation}\label{q}
 f(\zeta_k) = \lim_{\substack{z\to\zeta_k\\  \sphericalangle}}f(z) = p(\zeta_k),
 \end{equation}
 where $f = a_1 \widetilde{f} + p$ with $\widetilde{f} \in H^2$ and $p \in \P_{N - 1}$. In particular, 
 \begin{equation}\label{eq:E0b-rationnel}
 E_0(b)=\{\zeta_k:1\leq k\leq n\}. 
 \end{equation}
 Finally, let us mention that when $b \in \operatorname{ball}(H^{\infty})$ is a rational function and not a finite Blaschke product, then $\mathfrak M(\HH(b))=H^\infty\cap \HH(b)$.  See \cite{MR3967886}. 
 
 \par\smallskip
 \subsection{A description of $\HH(b)$ when $b=(1+I)/2$, with $I$ an inner function}
 There is another situation where we have an explicit description of the space \(\HH(b)\): this is when \(b=\frac{1+I}{2}\) and \(I\) is an inner function with \(I\not\equiv 1\). In this case, \(b\) is a non-extreme point of \(\textrm{ball}(H^{\infty})\), and its Pythagorean mate (up to a unimodular constant) is \(a=\frac{1-I}{2}\). Moreover, \((a,b)\) satisfies (HCR), since 
\(|a|^{2}+|b|^{2}\geq\frac{1}{2}\) on $\mathbb D$. In particular \(\HH(b)=\mathcal M(\bar a)\), with equivalent norms.
\par\smallskip
Under the assumption that $I(0)\not = 0$, it is proved in \cite{MR3850543} that 
\begin{equation}\label{eq:orthogonalite-hb-ma-KI}
\HH(b)=\mathcal M(a)\stackrel{\perp}{\oplus}_b  K_I,
\end{equation}
where the direct sum \(\stackrel{\perp}{\oplus}_b\)  is orthogonal with respect to the $\HH(b)$ norm. 
In particular, every $f\in\HH(b)$ can be written in a unique way as 
\begin{equation}\label{decomposition2-hb}
f=(1-I)g_1+g_2,\qquad \mbox{with }g_1\in H^2\mbox{ and }g_2\in K_I.
\end{equation}
It turns out that the same proof holds without any assumption on the value of $I(0)$. For completeness's sake, we present it in Lemma \ref{machin} below. 
We also give an equivalent norm on $\HH(b)$ analogue to (\ref{eq:norm in h(b)}).

\begin{Lemma}\label{machin}
 Let \(I\) be an inner function with \(I\not\equiv 1\), and let \(b=(1+I)/2\). Then the following assertions hold:
 \begin{enumerate}
  \item [\emph{(i)}] \(\HH(b)=(1-I)H^{2}\stackrel{\perp}{\oplus}_bK_{I}\), where 
  \(\stackrel{\perp}{\oplus}_b\) denotes an orthogonal direct sum in \(\HH(b)\);
  \item[\emph{(ii)}] if for \(f=(1-I)g_{1}+g_{2}\in\HH(b)\), \(g_{1}\in H^{2}\), \(g_{2}\in K_{I}\), we define
  \[
|||f|||_{b}^{2}=||g_{1}||_{2}^{2}+||g_{2}||_{2}^{2},
\]
then \(|||\,.\,|||_b\) is a norm on \(\HH(b)\) which is equivalent to \(||\,.\,||_{b}\).
 \end{enumerate}
\end{Lemma}

\begin{proof}
 (i) We have \(\HH(b)=\mathcal{M}(\bar{a})\) with equivalent norms, where \(a=\frac{1-I}{2}\) is the Pythago\-rean mate of \(b\). Also,
 \(
\frac{\bar{a}}{a}=\frac{1-\bar{I}}{1-I}=-\bar{I}
\)
a.e. on \(\T\), and thus \(T_{\bar{a}/a}=-T_{\bar{I}}\). Hence \(\ker T_{\bar{a}/a}=\ker T_{\bar{I}}=K_{I}\). Moreover, \(T_{a/\bar{a}}=-T_{I}\) has closed range, and thus 
\begin{align}
 H^{2}=\textrm{Ran}(T_{I})\stackrel{\perp}{\oplus}\ker(T_{I}^{*})=\textrm{Ran}(T_{I})\stackrel{\perp}{\oplus}\ker(T_{\bar{I}})
 &=T_{a/\bar{a}}H^{2}\stackrel{\perp}{\oplus}K_{I}\label{Eq 1}
\end{align}
(the sign $\stackrel{\perp}{\oplus}$ denotes here an orthogonal direct sum in $H^2$).
Using now the fact that \(T_{\bar{a}}\) is an isometry from \(H^{2}\) onto \(\mathcal M(\bar{a})=T_{\bar{a}}H^{2}\) (equipped with the range norm), applying \(T_{\bar{a}}\) to the equation (\ref{Eq 1}), and using the identity \(T_{\bar{a}}\,T_{a/\bar{a}}=T_{a}\), we obtain
\[
\HH(b)=\mathcal M(\bar{a})=\mathcal{M}(a)\stackrel{\perp}{\oplus}_{\,\bar{a}}T_{\bar{a}}\,K_{I},
\]
where the notation \(\stackrel{\perp}{\oplus}_{\,\bar{a}}\) represents an orthogonal direct sum with respect to the range norm on \(\mathcal M(\bar{a})\). Since 
\(T_{\bar{I}}\,K_{I}=\{0\}\) and \(\bar{a}=(1-\bar{I})/2\), we have \(T_{\bar{a}}\,K_{I}=(Id-T_{\bar{I}})\,K_{I}=K_{I}\), and so
\begin{equation}\label{Eq 2}
 \HH(b)=\mathcal M(\bar{a})=\mathcal{M}(a)\stackrel{\perp}{\oplus}_{\,\bar{a}}K_{I}=(1-I)H^{2}\stackrel{\perp}{\oplus}_{\,\bar{a}}K_{I}.
\end{equation}
It now remains to prove that the direct sum in this decomposition of \(\HH(b)\) is in fact orthogonal with respect to the \(\HH(b)\) norm.
\par\medskip
Let \(f\in H^{2}\) and \(g\in K_{I}\). Our aim is to show that \(\langle{(1-I)f},{g}\rangle_b=0\). Note that
\[
T_{\bar{b}}\,g=T_{(1+\bar{I})/2}\,g=\dfrac{1}{2}\,g=T_{\bar{a}}\,g
\]
from which it follows that 
\begin{equation}\label{Eq 3}
 g^{+}=g.
\end{equation}
Moreover, since \(\bar{b}\,a=-\bar{a}\,b\) a.e. on \(\T\), we have 
\[
T_{\bar{b}}\,\bigl ((1-I)f \bigr)=P_{+}\,(2\bar{b}\,af)=-P_{+}(2\bar{a}\,bf)=T_{\bar{a}}\,(-2bf),
\]
whence we get
\begin{equation}\label{Eq 4}
 \bigl ((1-I)f \bigr)^{+}=-2bf=-(1+I)f.
\end{equation}
By (\ref{Inconnu}), (\ref{Eq 3}) and (\ref{Eq 4}), it follows that
\begin{align*}
 \langle{(1-I)f},{g}\rangle_{b}&=\langle{(1-I)f},{g}\rangle_{2}-\langle{2bf},{g}\rangle_{2}
 =\langle{(1-I-2b)f},{g}\rangle_{2}
 =-2\,\langle{If},{g}\rangle_b=0
\end{align*}
because \(g\) belongs to \(K_{I}\).
\par\smallskip
(ii) Since \(\HH(b)=(1-I)H^{2}\stackrel{\perp}{\oplus}_{\,b}K_{I}\), we have
\[
||(1-I)g_{1}+g_{2}||_{b}^{2}=||(1-I)g_{1}||_{b}^{2}+||g_{2}||_{b}^{2},\qquad g_{1}\in H^2,\ g_{2}\in K_{I}.
\]
But observe that by (\ref{Inconnu}) and (\ref{Eq 4}) we have
\[
||(1-I)g_{1}||_{b}^{2}=||(1-I)g_{1}||_{2}^{2}+||(1+I)g_{1}||_{2}^{2}=4||g_{1}||_{2}^{2},
\]
while we get from (\ref{Inconnu}) and (\ref{Eq 3}) that \(||g_{2}||_{b}^{2}=2||g_{2}||_{2}^{2}\). Thus
\[
||(1-I)g_{1}+g_{2}||_{b}^{2}=4||g_{1}||_{2}^{2}+2||g_{2}||_{2}^{2},
\]
and from this the norm \(|||\,.\,|||_{b}\) is easily seen to be equivalent to \(||\,.\,||_{b}\).
\end{proof}

\par\smallskip
The next result is an analogue of \eqref{q} for the case where $b=(1+I)/2$ with respect to decomposition \eqref{decomposition2-hb}. 

 \begin{Lemma}\label{lem:existence-limite-radiale}
 Let $I$ be an inner function, $I\not\equiv 1$, and let $b=(1+I)/2$. Let $\zeta\in E_0(I)$ be such that $I(\zeta)=1$. Then $\zeta\in E_0(b)$. Moreover, if $f=(1-I)g_1+g_2$, $g_1\in H^2$ and $g_2\in K_I$, then $f(\zeta)=g_2(\zeta)$.
 \end{Lemma}
 \begin{proof}
 As mentioned above, since $\zeta\in E_0(I)$ the function $g_2$ has a non-tangential limit at the point $\zeta$. Thus it remains to prove that $(1-I)g_1$ has a zero non-tangential limit at $\zeta$. To this purpose,  write for $z\in\mathbb D$
 \[
 \begin{aligned}
 (1-I(z))g_1(z)=&\frac{1-\overline{I(\zeta)}I(z)}{1-\overline{\zeta}z}(1-\overline{\zeta}z)g_1(z)\\
 =&k_{\zeta}^I(z) \,\overline{\zeta}\,(\zeta-z)g_1(z)\\
 =&\langle k_{\zeta}^I,k_{z}^I\rangle_2  \,\overline{\zeta}\,(\zeta-z)g_1(z).
 \end{aligned}
 \] 
 Now, since $\zeta\in E_0(I)$, $k_{z}^I$ tends weakly to $k_\zeta^I$ as $z$ tends to $\zeta$ non-tangentially. Hence 
 \[
 \lim_{\substack{z\to \zeta\\ \sphericalangle}}\langle k_{\zeta}^I,k_{z}^I\rangle_2=\|k_\zeta^I\|_2^2<\infty.
 \]
 Moreover, using the estimate \eqref{gggbigggooooo}, we obtain that
 \[
 \lim_{\substack{z\to \zeta\\ \sphericalangle}}(\zeta-z)g_1(z)=0,
 \]
from which it follows that
\[
 \lim_{\substack{z\to \zeta\\ \sphericalangle}}(1-I(z))g_1(z)=0.
\]
 \end{proof}
 In the case where $b=(1+I)/2$ and $I$ is an inner function with $I\not\equiv 1$, there is no complete characterisation of multipliers for $\HH(b)$. Nevertheless, we have at our disposal a sufficient condition which will be useful for our study of cyclicity. Before stating this result (Lemma \ref{Lem:multiplier-b-1+I}) on multipliers, we recall a well-known property of model spaces, of which we provide a proof for completeness's sake.
 
 \begin{Lemma}\label{Lem:model-space-KI.KIcontenu dans KI^2}
 Let $I$ be an inner function and let $f\in K_I$ and $g\in K_I\cap H^\infty$. Then $fg\in K_{I^2}$. 
 \end{Lemma}
 
 \begin{proof}
Using that $K_I=H^2\cap I \overline{zH^2}$, write $f=I\overline{z\widetilde{f}}$ and $g=I\overline{z\widetilde{g}}$, with $\widetilde{f}, \widetilde{g}\in H^2$. Since $g\in H^\infty$, we indeed have $|\widetilde{g}|=|g|\in L^\infty(\mathbb T)$, and thus $\widetilde{g}\in H^\infty$. Moreover, $f g \in H^2$, and
 \[
 f g=I^2 \overline{z^2 \widetilde{f}\widetilde{g}},
 \]
 whence $f g\in H^2\cap I^2\overline{zH^2}=K_{I^2}$.
\end{proof}

 In the case where $b=(1+I)/2$, the de Branges-Rovnyak space $\mathcal H(b)$ contains a sequence of model spaces.
 \begin{Lemma}\label{Lem:sequence-model-spaces-contained-in-DBR}
Let $I$ be an inner function, $I\not\equiv 1$, and let $b=(1+I)/2$.  Then the following assertions hold:
\begin{enumerate}
\item the function $I$ is a multiplier of $\HH(b)$; 
\item for every $n\geq 1$, $K_{I^n}\subset \HH(b)$.
\end{enumerate} 
 \end{Lemma}
 
 \begin{proof}
 (a): Let $f\in\HH(b)$. According to \eqref{decomposition2-hb}, we can decompose $f$ as $f=(1-I)g_1+g_2$ with $g_1\in H^2$ and $g_2\in K_I$. Then
 \[
 If=(1-I)(Ig_1)+I g_2=(1-I)(Ig_1-g_2)+g_2
 \]
 and $Ig_1-g_2\in H^2$ and $g_2\in K_I$. Thus, using one more time \eqref{decomposition2-hb}, it follows that  $If\in\HH(b)$.
 
 (b): We argue by induction. For $n=1$, the property follows from Lemma~\ref{machin}. Assume that for some $n\geq 1$, $K_{I^n}\subset \HH(b)$. It is known that $K_{I^{n+1}}=K_I\oplus I K_{I^n}$. See \cite[Lemma 5.10]{MR3526203}. The conclusion now follows from the induction assumption and (a).
\end{proof}

 Here is now our sufficient condition for $f\in\HH(b)$ to be a multiplier of $\HH(b)$.
 
 \begin{Lemma}\label{Lem:multiplier-b-1+I}
 Let $I$ be an inner function, $I\not\equiv 1$, and let $b=(1+I)/2$. Assume that $f$ decomposes as $f=(1-I)g_1+g_2$, with $g_1\in H^\infty$ and $g_2\in H^\infty\cap K_I$. Then $f\in\mathfrak M(\HH(b))$. 
 \end{Lemma}
 
 \begin{proof}
 We need to show that for every $\varphi\in\HH(b)$, we have $\varphi f\in\HH(b)$. According to \eqref{decomposition2-hb}, write $\varphi=(1-I)\varphi_1+\varphi_2$, with $\varphi_1\in H^2$ and $\varphi_2\in K_I$. Then
 \[
 \varphi f=(1-I)\varphi_1 f+\varphi_2 f.
 \]
 Since $f\in H^\infty$, $\varphi_1 f\in H^2$, and so the first term $(1-I)\varphi _1 f$ belongs to $(1-I)H^2$ which is contained in $\HH(b)$. Thus it remains to prove that $\varphi_2 f\in\HH(b)$. In order to deal with this term, write
 \[
 \varphi_2 f=(1-I)\varphi_2 g_1+g_2\varphi_2,
 \]
 and as before, since $g_1\in H^\infty$, the term $(1-I)\varphi _2 g_1$ belongs to $(1-I)H^2$, and so to $\HH(b)$. It remains to prove that $g_2\varphi_2\in\HH(b)$. Lemma~\ref{Lem:model-space-KI.KIcontenu dans KI^2} implies that $g_2\varphi_2\in K_{I^2}$, and the conclusion follows now directly from Lemma~\ref{Lem:sequence-model-spaces-contained-in-DBR}.
 
 \end{proof}

\section{Some basic facts on cyclic vectors for the shift operator}\label{sec3}
Recall that if $T$ is a bounded operator on a Hilbert space $\HH$, then a vector $f\in\HH$ is said to be cyclic for $T$ if the linear span of the orbit of $f$ under the action of $T$ is dense in $\HH$, i.e. if
\[
\mbox{Span}(T^n f:n\geq 0)=\overline{\{p(T)f:p\in\mathbb C[X]\}}=\HH.
\] 
When $T=S_b$ is the shift operator on $\HH(b)$, we have $p(S_b)f=pf$ for every $f\in\HH(b)$ and every polynomial $p\in\C[X]$. Thus a function $f\in\HH(b)$ is cyclic for $S_b$ if and only if 
\[
\overline{\{pf:p\in\mathbb C[X]\}}=\HH(b).
\]
In fact, it is sufficient to approximate the constant function $1$ by elements of the form $pf$, $p\in\C[X]$, to get that $f$ is cyclic for $S_b$.

\begin{Lemma}\label{Lem:cyclicite-constant1}
Let $b$ a non-extreme point in $\operatorname{ball}(H^{\infty})$ and $f\in\HH(b)$. Then the following assertions are equivalent:
\begin{enumerate}
\item $f$ is cyclic for $S_b$.
\item There exists a sequence of polynomials $(p_n)_n$ such that 
\[
\|p_nf-1\|_b\to 0,\mbox{ as }n\to \infty.
\]
\end{enumerate}
\end{Lemma}

\begin{proof}
Follows immediately from the density of polynomials in $\HH(b)$ and the boundedness of $S_b$. 
\end{proof}

The general meaning of our next result is that the set of zeros of a cyclic vector $f\in\HH(b)$ for $S_b$ cannot be too large.  
\begin{Lemma}\label{Lem1:CS}
Let $b$ a non-extreme point in $\operatorname{ball}(H^{\infty})$ and $f\in\HH(b)$. Assume that $f$ is cyclic for $S_b$. Then we have the following properties:
\begin{enumerate}
\item $f$ is outer;
\item for every $\zeta\in E_0(b)$, $f(\zeta)\neq 0$. 
\end{enumerate}
\end{Lemma}

\begin{proof}
(a) Since $f$ is cyclic for $S_b$, there exists a sequence of polynomials $(p_n)_n$ such that 
\begin{equation}\label{eq-cyclicity-1-approximated}
\|p_nf-1\|_b\to 0,\mbox{ as }n\to \infty.
\end{equation}
Now, using the fact that $\HH(b)$ is contractively contained into $H^2$, we get that 
\[
\|p_nf-1\|_2\to 0,\mbox{ as }n\to \infty.
\]
That proves that $f$ is cyclic for $S$ in $H^2$, and so $f$ is outer by Beurling's theorem. 

\par\smallskip

(b) Since the functional $f\longmapsto f(\zeta)$ is bounded on $\HH(b)$ for every $\zeta\in E_0(b)$, we deduce  from \eqref{eq-cyclicity-1-approximated} that
\[
|p_n(\zeta)f(\zeta)-1|\to 0,\mbox{ as }n\to \infty
\]
for every $\zeta\in E_0(b)$.
This property implies directly that $f(\zeta)\neq 0$. 
\end{proof}

We will encounter in the sequel of the paper some situations where the converse of Lemma~\ref{Lem1:CS} is also true, i.e. where conditions (a) and (b) of Lemma \ref{Lem1:CS} give a necessary and sufficient condition for a function $f\in\HH(b)$ to be cyclic.
\par\smallskip
We now provide some elementary results concerning cyclic functions for $S_b$.

\begin{Lemma}\label{lem:multi-inver}
Let $b$ a non-extreme point in $\operatorname{ball}(H^{\infty})$. Suppose that $f\in\mathfrak M(\HH(b))$ and that $1/f\in \HH(b)$. Then $f$ is cyclic for $S_b$. 
\end{Lemma}
\begin{proof}
Using \eqref{eq:density-polynomial}, we see that there exists a sequence of polynomials $(p_n)_n$ such that 
\[
\|p_n-f^{-1}\|_b\to 0,\mbox{ as }n\to\infty. 
\]
Now, since $f\in\mathfrak M(\HH(b))$, the multiplication operator by $f$ is bounded on $\HH(b)$, and thus we get that 
\[
\|p_nf-1\|_b\to 0,\mbox{ as }n\to \infty,
\]
which by Lemma~\ref{Lem:cyclicite-constant1} implies that $f$ is cyclic for $S_b$.
\end{proof}

In the following result, the set $\mbox{Hol}(\overline{\mathbb D})$ denotes the space of analytic functions in a neighborhood of the closed unit disc $\overline{\mathbb D}$.
\begin{Corollary}\label{cor:cyclicity-holvoisinagedeDbar}
Let $b$ a non-extreme point in $\operatorname{ball}(H^{\infty})$. Let $f\in \mbox{Hol}(\overline{\mathbb D})$ and assume that $\inf_{\overline{\mathbb D}}|f|>0$. Then $f$ is cyclic for $S_b$. 
\end{Corollary}

\begin{proof}
When $b$ is a non-extreme point in $\operatorname{ball}(H^{\infty})$, we have $\mbox{Hol}(\overline{\mathbb D})\subset \mathfrak M(\HH(b))$. See \cite[Theorem 24.6]{FM2}. 
Hence $f\in  \mathfrak M(\HH(b))$.  Moreover, the conditions on $f$ also imply that $1/f \in \mbox{Hol}(\overline{\mathbb D})$. In particular, $1/f\in\HH(b)$. It remains to apply Lemma~\ref{lem:multi-inver} in order to get that $f$ is cyclic.
\end{proof}

\begin{Lemma}\label{Lem:product-multiplicateur}
Let $f_1,f_2\in\mathfrak M(\HH(b))$. Then the following assertions are equivalent:
\begin{enumerate}
\item the product function $f_1f_2$ is cyclic for $S_b$;
\item each of the functions $f_1$ and $f_2$ is cyclic for $S_b$.
\end{enumerate}
\end{Lemma}

\begin{proof}
$(a)\implies (b)$: Assume that $f_1f_2$ is cyclic. By symmetry, it suffices to prove that $f_1$ is cyclic. Let $\varepsilon>0$. There exists a polynomial $q$ such that $\|qf_1f_2-1\|_b\leq \epsilon$. Now since the polynomials are dense in $\HH(b)$, we can also find a polynomial $p$ such that $$\|f_2q-p\|_b\leq \frac{\varepsilon}{\|f_1\|_{\mathfrak M(\HH(b))}}\cdot$$ Thus
\[
\begin{aligned}
\|pf_1-1\|_b\leq & \|pf_1-f_1f_2q\|_b+\|f_1f_2q-1\|_b
\leq \|f_1\|_{\mathfrak M(\HH(b))} \|p-f_2q\|_b+\varepsilon
\leq 2\varepsilon,
\end{aligned}
\]
which proves that $f_1$ is cyclic.

$(b)\implies (a)$: Assume that $f_1$ and $f_2$ are cyclic for $\HH(b)$. Let $\varepsilon>0$. There exists a polynomial $p$ such that $\|pf_1-1\|_b\leq \epsilon$. On the other hand, there is also a polynomial $q$ such that $$\|qf_2-1\|_b\leq \frac{\varepsilon}{\|pf_1\|_{\mathfrak M(\HH(b))}}\cdot$$ Now we have 
\[
\begin{aligned}
\|pqf_1f_2-1\|_b\leq & \|pqf_1f_2-pf_1\|_b+\|pf_1-1\|_b
&\leq \|pf_1\|_{\mathfrak M(\HH(b))} \|qf_2-1\|_b+\varepsilon
&\leq 2\varepsilon.
\end{aligned}
\]
Hence the function $f_1f_2$ is cyclic. 
\end{proof}

Our next result is motivated by the Brown--Shields conjecture and the work \cite{MR3475457} for Dirichlet type spaces $\mathcal D(\mu)$. Indeed, 
let $\mu$ be a positive finite measure on $\T$, and let $\mathcal D(\mu)$ be the associated Dirichlet space (i.e. the space of holomorphic functions on $\D$ whose derivatives are square-integrable when weighted against the Poisson integral of the measure $\mu$). It is shown in \cite{MR3110499}, \cite{MR3390195} that in some cases, Dirichlet spaces and de Branges-Rovnyak spaces are connected. More precisely, let $b\in \textrm{ball}(H^{\infty})$ be a rational function (which is not a finite Blaschke product), and let $a$ be its pythagorean mate. Let also $\mu$ be a positive finite measure on $\T$. Then $\mathcal D(\mu)=\HH(b)$ with equivalent norms if and only if the zeros of $a$ on $\T$ are all simple, and the support of $\mu$ is exactly the set of these zeros \cite{MR3110499}.
In the context of Dirichlet spaces, the authors of \cite{MR3475457}
prove the Brown--Shields conjecture when the measure $\mu$ has countable support, using two notions of capacity (which they denote $c_{\mu}(F)$ and $c_{\mu}^a(F)$ respectively) and showing that they are comparable: $c_{\mu}(F)\le c_{\mu}^a(F)\le 4 \,c_{\mu}(F)$ for every $F\subset \mathbb T$ (\cite[Lemma 3.1]{MR3475457}).
In the same spirit, we introduce the following notions of {capacity} in $\HH(b)$-spaces. For a set $F\subset \mathbb T$, we define  $c_1(F)$ and $c_2(F)$ as
\[
c_1(F)=\inf\{\|f\|_b: f\in\HH(b),\,|f|\geq 1\mbox{ a.e. on a neighborhood of }F\},
\]
and 
\[
c_2(F)=\inf\{\|f\|_b: f\in\HH(b),\,|f|=1\mbox{ a.e. on a neighborhood of }F\}.
\]
Observe  that $c_1(F)\leq c_2(F)$. We do not know if $c_1(F)$ and $c_2(F)$ are comparable in general in our context of de Branges-Rovnyak spaces.
\par\smallskip
% 
% When $\mu$ is a positive finite measure on $\T$ and $\mathcal D(\mu)$ is the associated Dirichlet space (i.e. the space of holomorphic functions on $\D$ whose derivatives are square-integrable when weighted against the Poisson integral of the measure $\mu$), the authors of \cite{MR3475457} introduce similar notions of capacities (which they denote $c_{\mu}(F)$ and $c_{\mu}^a(F)$ respectively), and show that they are comparable: $c_{\mu}(F)\le c_{\mu}^a(F)\le 4 \,c_{\mu}(F)$ for every $F\subset \mathbb T$ (\cite[Lemma 3.1]{MR3475457}).
% We do not know if a similar result holds in our context of de Branges-Rovnyak spaces. However, it is shown in \cite{MR3110499} that in some cases, Dirichlet spaces and de Branges-Rovnyak spaces are connected. More precisely, let $b\in \textrm{ball}(H^{\infty})$ be a rational function (which is not a finite Blaschke product), and let $a$ be its pythagorean mate. Let also $\mu$ be a positive finite measure on $\T$. Then $\mathcal D(\mu)=\HH(b)$ with equivalent norms if and only if the zeros of $a$ on $\T$ are all simple, and the support of $\mu$ is exactly the set of these zeros \cite{MR3110499}.

Our next result should be compared to \cite[Lemma 3.2]{MR3475457}.

\begin{Theorem}\label{Thm:capacity}
Let $b$ a non-extreme point in $\operatorname{ball}(H^{\infty})$ and $\zeta\in\mathbb T$.
Consider the following assertions:
\begin{enumerate}
\item $z-\zeta$ is not cyclic for $S_b$;
\item $\zeta\in E_0(b)$;
\item $c_1(\{\zeta\})>0$;
\item $c_2(\{\zeta\})>0$.
\end{enumerate}
Then $(a)\iff(b)$, $(b)\implies (d)$ and $(c)\implies (a)$.
\end{Theorem}

\begin{proof}  $(b)\implies (a)$: follows immediately from Lemma~\ref{Lem1:CS}.
\par\smallskip
$(a)\implies (b)$: our assumption (a) exactly means that 
\[
[z-\zeta]:=\mbox{Span}((z-\zeta)z^n:n\geq 0)\subsetneq \HH(b).
\] 
Denote by $\pi$ the orthogonal projection from $\HH(b)$ onto $[z-\zeta]^\perp$. First note that $\pi(1)\neq 0$, otherwise we would have $1\in [z-\zeta]$ and then the function $z-\zeta$ would be cyclic for $S_b$,  which is a contradiction. 

Let us now prove that $[z-\zeta]^\perp=\mathbb C \pi(1)$. For every $g\in [z-\zeta]^\perp$ and every $n\geq 0$, we have
\[
0=\langle g,(z-\zeta)z^n\rangle_b = \langle g,z^{n+1}\rangle_b-\overline{\zeta} \langle g,z^n\rangle_b.
\]
From this, we immediately get that 
\begin{equation}\label{point}
 \langle g,z^n\rangle_b={\overline{\zeta}}^n \langle g,1\rangle_b, \quad n\ge 0.
\end{equation}
This implies that $\langle \pi(1),1\rangle_b\neq 0$ (otherwise, by (\ref{point}) we would have that $\pi(1)$ is orthogonal to $z^n$ for every $n\ge0$, which implies that $ \pi(1)=0$). Secondly, if we define $c:=\frac{\langle g,1\rangle_b}{\langle \pi(1),1\rangle_b}$, then we have $\langle g-c\pi(1),z^n\rangle_b=0$ for every $n\geq 0$. By the density of polynomials in $\HH(b)$, we deduce that $g=c\pi(1)$, which proves that $[z-\zeta]^\perp$ is of dimension $1$, generated by $\pi(1)$. 

Now consider the continuous linear functional $\rho:\mathbb C\pi(1)\longrightarrow \mathbb C$ defined by $\rho(\alpha\pi(1))=\alpha$ for every $\alpha\in\C$. Let us check that for every $n\geq 0$, 
\begin{equation}\label{eq13EZ:Thm:capacity}
(\rho\circ\pi)(z^n)=\zeta^n. 
\end{equation}
For $n=0$, this is true by definition. Assume that \eqref{eq13EZ:Thm:capacity} is satisfied for some integer $n\geq 0$. Then, 
\[
(\rho\circ\pi)(z^{n+1})=(\rho\circ\pi)(z^n(z-\zeta))+\zeta (\rho\circ\pi)(z^n)=\zeta^{n+1}.
\]
By induction, we deduce \eqref{eq13EZ:Thm:capacity} and by linearity, for any polynomial $p$, we have $(\rho\circ \pi)(p)=p(\zeta)$. Now, using the continuity of $\rho$ and $\pi$, we obtain that there exists a constant $C>0$ such that 
\[
|p(\zeta)|\leq C \|p\|_b,\qquad \mbox{for any polynomial }p\in\C[X].
\]
Denote by $L_\zeta$ the linear functional defined on $\mathbb C[X]$ by $L_\zeta(p)=p(\zeta)$, $p\in\mathbb C[X]$. Then $L_\zeta$ is continuous on
$\C[X]$ endowed with the topology of
$\HH(b)$. Hence it extends to a continuous linear map on $\HH(b)$. By the Riesz representation theorem, there exists a unique vector $h_\zeta\in\HH(b)$, $h_\zeta\neq 0$,  such that 
\[
p(\zeta)=L_\zeta(p)=\langle p,h_\zeta\rangle_b, \qquad \mbox{for any polynomial }p\in\C[X].
\]
Now, note that for any polynomial $p$, we have 
\[
\langle p,S_b^* h_\zeta\rangle_b=\langle zp,h_\zeta \rangle_b=\zeta p(\zeta)=\langle p,\overline{\zeta}h_\zeta\rangle_b,
\]
whence, using \eqref{eq:density-polynomial}, $S_b^*h_\zeta=\overline{\zeta}h_\zeta$. In particular, $\overline{\zeta}$ belongs to the point spectrum of $S_b^*$. But by \eqref{point-spectrum-boundary}, this implies that $b$ has an angular derivative at $\zeta$, which is equivalent to the property that $\zeta\in E_0(b)$. Note that the function $h_\zeta$ is in fact the reproducing kernel $k_\zeta^b$ of $\HH(b)$ at the point $\zeta$.
\par\smallskip
$(b)\implies (d)$: assume now that $\zeta\in E_0(b)$. Let $f\in\HH(b)$ be such that $|f|=1$ a.e. on a neighborhood $\mathcal O$ of $\zeta$. Let us consider the inner-outer factorisation of $f=f_i f_o$, where $f_i$ is the inner part and $f_o$ the outer part of $f$. Since by definition $|f_i|=1$ a.e. on $\mathbb T$, we have $|f_o|=1$ a.e. on $\mathcal O$. Moreover, $f_o\in\HH(b)$ and $\|f_o\|_b\leq \|f\|_b$. Indeed, $f_o=T_{\bar f_i}f$, where $T_{\bar f_i}$ is the Toeplitz operator with symbol $\bar f_i$ and $\HH(b)$ is invariant with respect to co-analytic Toeplitz operators. Furthermore, 
\[
\|f_o\|_b=\|T_{\bar f_i}f\|_b\leq \|f_i\|_\infty \|f\|_b=\|f\|_b.
\]
See \cite[Theorem 18.13]{FM2}. Since $f_o$ is outer and $\log|f_o|=0$ a.e. on $\mathcal O$, we have 
\[
f_o(z)=\lambda\exp\left(\int_{\mathbb T\setminus\mathcal O}\frac{\xi+z}{\xi-z}\log|f_o(\xi)|\,dm(\xi)\right),
\]
for some constant $\lambda\in\mathbb T$. Hence $f_o$ is analytic in a neighborhood of $\zeta$ and in particular, we deduce that $|f_o(\zeta)|=1$. Using now the fact that $\zeta\in E_0(b)$,  we know that there exists a constant $C>0$ such that  $|g(\zeta)|\leq C \|g\|_b$ for every $g\in\HH(b)$. Hence
\[
1=|f_o(\zeta)|\leq C \|f_o\|_b\leq C \|f\|_b 
\]
for every function $f\in\HH(b)$ such that $|f|=1$ a.e. on a neighborhood $\mathcal O$ of $\zeta$. We deduce that $c_2(\{\zeta\})\geq C^{-1}>0$. 
\par\smallskip
$(c)\implies (a)$: by contradiction, assume that $z-\zeta$ is cyclic for $S_b$. Then, for every $\varepsilon >0$, we can find a polynomial $q$ such that $\|q(z-\zeta)-1\|_b\leq \varepsilon$. Note that the value of the polynomial $q(z-\zeta)-1$ at $\zeta$ is $-1$. So by continuity, we can find a neighborhood $\mathcal O$ of $\zeta$ on $\mathbb T$ such that $|q(z-\zeta)-1|\geq 1/2$ on $\mathcal O$. Hence $|2(q(z-\zeta)-1)|\geq 1$ on $\mathcal O$ and by definition of $c_1(\{\zeta\})$, we obtain that 
\[
c_1(\{\zeta\})\leq 2\|q(z-\zeta)-1\|_b\leq 2 \varepsilon.
\]
Since this is true for every $\varepsilon>0$, we deduce that $c_1(\{\zeta\})=0$, which contradicts $(c)$. 
\end{proof}

If we knew that $c_1(F)$ and $c_2(F)$ were comparable, assertions (a) to (d) in Theorem \ref{Thm:capacity} would be equivalent. This motivates the following question:

\begin{question}
(i) Does there exist $\kappa>0$ such that  $c_2(F)\le \kappa\, c_1(F)$ for every $F\subseteq\T$? 

\noindent
(ii) Is it true that $c_1(F)>0$ if and only if $c_2(F)>0$?
\end{question}

\begin{Remark}
It can be easily seen from Theorem \ref{Thm:capacity} that the condition $\inf_{\overline{\mathbb D}}|f|>0$ in Corollary~\ref{cor:cyclicity-holvoisinagedeDbar} is not necessary for $f$ to be cyclic in $\HH(b)$. Indeed, let $b(z)=\frac{1+z}{2}S_{\delta_1}(z)$, where $S_{\delta_1}$ is the singular inner function associated to $\delta_1$, the Dirac measure at point $1$. See \eqref{E:dec-b-bso}. It is clear that
\[
\int_{\mathbb T}\frac{d\delta_1(\xi)}{|\zeta_0-\xi|^2}=\frac{1}{|\zeta_0-1|^2}=\infty
\]
when $\zeta_0=1$. Hence $1\notin E_0(b)$. Therefore, by Theorem~\ref{Thm:capacity}, the function $z-1$ is cyclic for $S_b$ while $\inf_{\overline{\mathbb D}}|z-1|=0$.
\end{Remark}

\begin{Corollary}
Let $b$ a non-extreme point in $\operatorname{ball}(H^{\infty})$. Let $p$ be a polynomial. 
The following assertions are equivalent: 
\begin{enumerate}
\item $p$ is cyclic for $S_b$.
\item $p(z)\neq 0$ for every $z\in\mathbb D\cup E_0(b)$. 
\end{enumerate}
\end{Corollary}

\begin{proof}
$(a)\implies (b)$: follows immediately from Lemma~\ref{Lem1:CS}. 

$(b)\implies (a)$: factorise the polynomial $p$ as $p(z)=c\prod_{j=1}^n (z-\zeta_j)$, where by definition the roots $\zeta_j$ belong to $\mathbb C\setminus (\mathbb D\cup E_0(b))$. On the one hand, if $|\zeta_j|>1$, then, according to Corollary~\ref{cor:cyclicity-holvoisinagedeDbar}, the function $z-\zeta_j$ is cyclic for $S_b$. On the other hand, if $|\zeta_j|=1$, then $\zeta_j\not\in E_0(b)$ and Theorem~\ref{Thm:capacity} implies that  the function $z-\zeta_j$ is also cyclic for $S_b$. Thus, for every $1\leq j\leq n$, the function $z-\zeta_j$ is cyclic and it follows from Lemma~\ref{Lem:product-multiplicateur} that $p$ itself is cyclic for $S_b$.  \end{proof}

This result can be slightly generalised:

\begin{Corollary}\label{Cor:cyclicite-fonction-holomorphe-dans-Dbar}
Let $b$ a non-extreme point in $\operatorname{ball}(H^{\infty})$. Let $f\in\mbox{Hol}(\overline{\mathbb D})$.
The following assertions are equivalent: 
\begin{enumerate}
\item $f$ is cyclic for $S_b$.
\item $f$ is outer and $f(\zeta)\neq 0$ for every $\zeta\in E_0(b)$. 
\end{enumerate}

\end{Corollary}

\begin{proof}
$(a)\implies (b)$: follows from Lemma~\ref{Lem1:CS}. 

$(b)\implies (a)$: since $f$ is outer and $f\in\mbox{Hol}(\overline{\mathbb D})$, $f$ does not vanish on the unit disc and has at most a finite number of zeros on $\mathbb T$ (otherwise by compactness and the uniqueness principle for holomorphic functions, $f$ would vanish identically). Let $\zeta_1,\zeta_2,\dots,\zeta_n$ be the (possible) zeros of $f$ on $\mathbb T$. Then there exists a function $g\in\mbox{Hol}(\overline{\mathbb D})$ with $\inf_{\overline{\mathbb D}}|g|>0$ such that
\[
f(z)=\prod_{j=1}^n (z-\zeta_j) g(z),\qquad z\in\overline{\mathbb D}.
\]
Our assumption implies that for every $1\leq j\leq n$, $\zeta_j\notin E_0(b)$, and thus by Theorem~\ref{Thm:capacity}, the function $z-\zeta_j$ is cyclic for $S_b$. Moreover, by Corollary~\ref{cor:cyclicity-holvoisinagedeDbar}, the function $g$ is also cyclic. Now it follows from Lemma~\ref{Lem:product-multiplicateur} that $f$ itself is cyclic for $S_b$. 
\end{proof}

\begin{Example}\label{example-rkhardycyclic}
Let $b$ a non-extreme point in $\operatorname{ball}(H^{\infty})$. For every $\lambda\in\mathbb D$, $k_\lambda$ is a cyclic vector for $S_b$. Indeed, it is clear that $k_\lambda(z)=(1-\overline{\lambda}z)^{-1}$ satisfies the conditions of Corollary~\ref{Cor:cyclicite-fonction-holomorphe-dans-Dbar}. Hence $k_\lambda$ is cyclic. \
\par\smallskip
In particular, by \eqref{eq:density-crk}, the set of cyclic vectors for $S_b$ spans a dense subspace in $\HH(b)$.
\end{Example}

\par\medskip
\section{The rational case}\label{sec4}

The main result of this section is a characterisation of cyclic functions for $S_b$ when $b$ is a rational function which is not a finite Blaschke product. As mentioned already in the Introduction, this result can be derived from the work \cite{MR4246975} by Luo -- Gu -- Richter, but we provide here an elementary proof, the ideas of which will turn out to be also relevant to the case where $b=(1+I)/2$ (see Section \ref{sec5} below). Note that Theorem \ref{Thm:rational-case} extends a result proved in \cite{MR3309352}
in the particular case where $b(z)=(1+z)/2$.

\begin{Theorem}\label{Thm:rational-case}
Let $b\in \operatorname{ball}(H^{\infty})$ and assume that $b$ is rational (but not a finite Blaschke product). Let $a_1$ be the associated polynomial given by \eqref{eq:definition of a}, and let $f\in \HH(b)$. Then the following assertions are equivalent:
\begin{enumerate}
\item $f$ is cyclic for $S_b$.
\item $f$ is an outer function and for every $1\leq k\leq n$, $f(\zeta_k)\neq 0$. 
\end{enumerate}
\end{Theorem}

\begin{proof}
$(a)\implies (b)$: according to \eqref{eq:E0b-rationnel}, we know that $E_0(b)=\{\zeta_k:1\leq k\leq n\}$. Hence this implication follows from Lemma~\ref{Lem1:CS}. 

$(b)\implies (a)$: according to \eqref{uUUiipPPS}, write $f=a_1\widetilde{f}+p$, where $\widetilde f\in H^2$ and $p\in\P_{N-1}$. By \eqref{q}, $p(\zeta_k)\neq 0$, $1\leq k\leq n$. 
Let now $r\in\P_{N-1}$ be the unique polynomial satisfying the following interpolation properties:  for every $1\leq k\leq n$,
% Let us define (by induction) the Hermite interpolating polynomial $r\in\P_{N-1}$ such that 
\[
r^{(j)}(\zeta_k)=\begin{cases}
\frac{1}{p(\zeta_k)}& \mbox{if }j=0 \\
-\frac{1}{p(\zeta_k)}\,{\sum_{\ell=0}^{j-1}\binom{j}{\ell}r^{(\ell)}(\zeta_k)p^{(j-\ell)}(\zeta_k)}& \mbox{if }1\leq j\leq m_{k}-1.
\end{cases}
\]
This polynomial $r$ can be constructed using Hermite polynomial interpolation, see for instance \cite[Chapter 1, E. 7]{MR1367960}.
By Leibniz's rule, we easily see that for every $1\leq k\leq n$ and $0\leq j\leq m_k-1$, we have $(rp-1)^{(j)}(\zeta_k)=0$. Hence $a_1$ divides the polynomial $rp-1$. In other words, there exists a polynomial $q$ such that $rp-1=a_1 q$. Using that $f$ is outer, we can find a sequence of polynomials $(q_n)_n$ such that $\|q_nf+r\widetilde{f}+q\|_2\to 0$ as $n\to \infty$. Define now a sequence of polynomials $(p_n)$ by $p_n=a_1q_n+r$, $n\geq 1$. Observe that 
\[
\begin{aligned}
p_nf-1=&(a_1q_n+r)f-1=a_1q_nf+r(a_1\widetilde{f}+p)-1\\
=&a_1(q_nf+r\widetilde{f})+rp-1=a_1(q_nf+r \widetilde{f}+q).
\end{aligned}
\]
Then it follows from \eqref{eq:norm in h(b)} that
\[
\vvvert p_nf-1\vvvert_b=\vvvert a_1(q_nf+r \widetilde{f}+q) \vvvert_b=\|q_nf+r \widetilde{f}+q\|_2\to 0\,\textrm{ as } n \to \infty.
\]
Therefore $f$ is cyclic for $S_b$.
 \end{proof}

\begin{Example}
Let $b(z)=\frac{1}{2}(1-z^2)$. Then it is proved in \cite{MR3503356} that $a(z)=c(z-i)(z+i)$, for some constant $c$. Thus, according to Theorem~\ref{Thm:rational-case}, a function $f\in\HH(b)$ is cyclic for $S_b$ if and only if $f$ is outer, $f(i)\neq 0$ and $f(-i)\neq 0$. 
\end{Example}

\section{The case where $b=(1+I)/2$}\label{sec5}
Our main result in this section is the following:

\begin{Theorem}\label{Thm:main-b-1+I}\label{sec5-th}
Let $I$ be an inner function, $I\not\equiv 1$, and assume that its Clark measure $\sigma_1$ associated to point $1$ (defined in \eqref{def-clark-measure}) is a discrete measure. Let $\{\zeta_n:n\geq 1\}=\{\zeta\in E_0(I): I(\zeta)=1\}$. Let $b=(1+I)/2$, and $f\in\HH(b)$ which we decompose according to \eqref{decomposition2-hb} as $f=(1-I)g_1+g_2$, with $g_1\in H^2$, $g_2\in K_I$. Assume that:
\begin{enumerate}
\item $g_1,g_2\in H^\infty$;
\item $f$ is outer;
\item we have 
\[
\sum_{n\geq 1}\frac{1}{|f(\zeta_n)|^2 |I'(\zeta_n)|}<\infty.
\]
\end{enumerate}
Then $f$ is cyclic for $S_b$. 
\end{Theorem}

\begin{proof}
The proof proceeds along the same lines as in the  rational case. 
According to Lemma~\ref{Lem:multiplier-b-1+I}, $f\in\mathfrak M(\HH(b))$ and by Lemma~\ref{lem:existence-limite-radiale} we have $f(\zeta_n)=g_2(\zeta_n)$, $n\geq 1$. 
\par\smallskip
\noindent
{\emph{First step}}: We claim that there exists a sequence of functions $(\psi_n)_n$ in $\HH(b)$ such that $\|\psi_n f-1\|_b\to 0$ as $n\to\infty$. 
\par\smallskip
In order to construct the sequence $(\psi_n)_n$,  let us first consider the function $r$ given by
\[
r=\sum_{n=1}^\infty \frac{1}{f(\zeta_n)}\frac{k_{\zeta_n}^I}{\|k_{\zeta_n}^I\|_2^2}\cdot
\]
Recall that by \eqref{eq:nice-formula-derivative-norme-kernel-clark-measure}, $\|k_{\zeta_n}^I\|_2^2=|I'(\zeta_n)|$. 
Combining this with condition (c) and the fact that the family $(k_{\zeta_n}^I/\|k_{\zeta_n}^I\|_2)_n$ forms an orthonormal basis of $K_I$ (since $\sigma_1$ is discrete, Clark's theorem holds true), we see that the series defining the function $r$ is convergent in $K_I$. In others words, $r\in K_I$ and $r(\zeta_n)=1/f(\zeta_n)=1/g_2(\zeta_n)$ for every $n\geq 1$. 
\par\smallskip
Let us now prove that $rg_2-1\in (1-I)H^2$. Observe that Lemma~\ref{Lem:model-space-KI.KIcontenu dans KI^2} implies that $rg_2\in K_{I^2}=K_I\oplus IK_I$. Hence there exist $\varphi_1,\varphi_2\in K_I$ such that $rg_2-1=\varphi_1+I\varphi_2-1$. Since $r(\zeta_n)g_2(\zeta_n)-1=0$, we have $\varphi_1(\zeta_n)+\varphi_2(\zeta_n)-1=0$ for every $n\geq 1$.  
Note that \((1-\overline{I(0)})^{-1}(1-\overline{I(0)}I)=((1-\overline{I(0)})^{-1} k_0^I \in K_{I}\) and
\[
(1-\overline{I(0)})^{-1}(1-\overline{I(0)}I(\zeta _{n}))=1\qquad\textrm{for every}\ n\ge 1.
\]
Since the family \(\bigl (k_{\zeta _{n}}^{I} \bigr)_{n} \) is complete in \(K_{I}\), and since \(\varphi _{1}(\zeta _{n})+\varphi _{2}(\zeta _{n})=1\) for every \(n\ge 1\), we deduce that
\[
\varphi _{1}+\varphi_{2}=(1-\overline{I(0)})^{-1}(1-\overline{I(0)}I).
\]
Hence
\begin{align*}
 rg_{2}-1&=\varphi _{1}+I\varphi _{2}-1
 =-\varphi _{2}+(1-\overline{I(0)})^{-1}(1-\overline{I(0)}I)+I\varphi _{2}-1\\
 &=(1-I)(-\varphi _{2})+(1-\overline{I(0)})^{-1}(1-\overline{I(0)}I)-1.
\end{align*}
Observe that
\[
(1-\overline{I(0)})^{-1}(1-\overline{I(0)}I)-1=(1-\overline{I(0)})^{-1}\overline{I(0)}(1-I),
\]
from which it follows that 
\[
rg_{2}-1=(1-I)(-\varphi _{2}+\overline{I(0)}(1-\overline{I(0)})^{-1}).
\]
This proves that \(rg_{2}-1\) belongs to \((1-I)H^{2}\). Write \(rg_{2}-1\) as 
\(rg_{2}-1=(1-I)g_{3}\), with \(g_{3}\in H^{2}\).
% But $\varphi_1+\varphi_2-1\in K_1$ (the property $I(0)=0$ is equivalent to $1\in K_I$) and the family of functions $(k_{\zeta_n}^I)_n$ is complete in $K_I$, whence $\varphi_1+\varphi_2-1=0$. In other words, $\varphi_2=1-\varphi_1$ and $rg_2-1=\varphi_1+I(1-\varphi_1)-1=(1-I)(\varphi_1-1)\in (1-I)H^2$. 
% 
% Write $rg_2-1=(1-I)g_3$, with $g_3\in H^2$. 
\par\smallskip
Using now that $f$ is outer, and that $rg_1\in H^2$ (as $g_1\in H^{\infty}$) we can find a sequence of polynomials $(q_n)_n$ such that $\|q_nf+rg_1+g_3\|_2\to 0$, as $n\to\infty$. We then define for each $n\geq 1$ a function $\psi_n$ as
\[
\psi_n:=(1-I)q_n+r,
\]
where $q_n$ and $r$ are defined above. Note that $\psi_n\in (1-I)H^2+K_I=\HH(b)$ and 
\[
\begin{aligned}
\psi_n f-1=&(1-I)q_nf+rf-1\\
=&(1-I)q_nf+(1-I)rg_1+rg_2-1\\
=&(1-I)(q_nf+rg_1+g_3).
\end{aligned}
\]
It follows from Lemma~\ref{machin} that there exists a positive constant $C$ such that
\[
\begin{aligned}
\|\psi_nf-1\|_b =&\|(1-I)(q_nf+rg_1+g_3)\|_b\\
\le& C\, \vvvert (1-I)(q_nf+rg_1+g_3) \vvvert_b\\
=& C\, \|q_nf+rg_1+g_3\|_2,
\end{aligned}
\]
from which it follows that $\|\psi_nf-1\|_b\to 0$ as $n\to\infty$. 
\par\smallskip
{\emph{Second step}}: Let us now prove that there exists a sequence of polynomials $(p_{n})_n$ such that $\|p_ {_n}f-1\|_b\to 0$ as $n\to\infty$.
\par\smallskip
By the density of polynomials in $\HH(b)$, we can find a sequence of polynomials $(p_{n})_n$ such that $\|p_{n}-\psi_n\|_b\to 0$ as $n\to\infty$. Now write
\[
\begin{aligned}
\|p_{n}f-1\|_b\leq& \|p_{n}f-\psi_n f\|_b+\|\psi_n f-1\|_b\\
\leq & \|f\|_{\mathfrak M(\HH(b)} \|p_{n}-\psi_n\|_b+\|\psi_n f-1\|_b,
\end{aligned}
\]
and by the choice of the sequence $(p_{n})_n$ and the first step, we get the conclusion of the second step. 
\par\smallskip
We finally conclude that $f$ is cyclic for $S_b$.
\end{proof}

\begin{Remark}
If $I$ is an inner function such that, for some $\alpha\in\mathbb T$, its Clark measure $\sigma_\alpha$ is a discrete measure and $I\not\equiv \alpha$, then we may apply Theorem~\ref{sec5-th} replacing $I$ by $\bar\alpha I$ and $b=(1+I)/2$ by $b=(1+\bar\alpha I)/2$.  
\end{Remark}

\begin{Example}
Let \(I=S_{\delta _{1}}\) be the inner function associated to the measure \(\delta _{1}\):
\[
I(z)=\exp\Bigl (-\dfrac{1+z}{1-z} \Bigr) .
\]
In this case we can compute explicitly the Clark basis of \(K_{I}\) associated to point $1$. We have \(E_1=\{\zeta\in E_0(I):I(\zeta)=1\}=\{\zeta _{n}\;;\;n\in\mathbb{Z}\}\) with
\[\zeta _{n}=\dfrac{2i\pi n-1}{2i\pi n+1}\quad\textrm{and}\quad I'(\zeta _{n})=-\dfrac{1}{2}(2i\pi n+1)^{2},\ n\in\mathbb{Z}.\]
Therefore, if \(f\in\mathcal{H}(b)\) is outer, with \(f=(1-I)g_{1}+g_{2}\), \(g_{1}\in H^{\infty}\), \(g_{2}\in K_{I}\cap H^{\infty}\), and if 
\[
\sum_{n\in\mathbb{Z}}\dfrac{1}{|f(\zeta _{n})|^{2}}\cdot\dfrac{1}{4n^{2}\pi ^{2}+1}<+\infty,
\]
then \(f\) is cyclic for \(S_{b}\). 
\end{Example}

\begin{Remark}
There exists a recipe to construct an inner function satisfying the hypothesis of Theorem~\ref{sec5-th}. Let $\sigma$ be a positive discrete measure on $\mathbb T$ and let $H\sigma$ be its Herglotz transform, 
\[
H\sigma(z)=\int_{\mathbb T}\frac{\xi+z}{\xi-z}\,d\sigma(\xi),\qquad z\in\mathbb D.
\]
We easily see that $H\sigma$ defines an analytic function on $\mathbb D$ and satisfies $\Re e (H\sigma(z))\geq 0$ for every $z\in\D$. Now define a function $I$ on $\mathbb D$ as 
\[
I=\frac{H\sigma-1}{H\sigma+1}\cdot
\]
Since $\Re e H\sigma\geq 0$, it is easy to check that $I\in H^\infty$ and $|I|\leq 1$. Moreover, for every $0<r<1$ and $\zeta\in\mathbb T$, we have
\begin{equation}\label{trick-construction-inner-function}
|I(r\zeta)|=\frac{(\Re e (H\sigma(r\zeta))-1)^2+(\Im m (H\sigma(r\zeta))^2}{(\Re e (H\sigma(r\zeta))+1)^2+(\Im{m}(H\sigma(r\zeta))^2}\cdot
\end{equation}
Since $\sigma$ is a singular measure, it is well-known that for almost all $\zeta\in\mathbb T$, we have
\[
\Re e(H\sigma(r\zeta))=\int_{\mathbb T}\frac{1-r^2}{|\xi-r\zeta|^2}\,d\sigma(\xi)\to 0\quad\mbox{as }r\to 1^-.
\]
See \cite[Corollary 3.4]{FM1}. Moreover, the radial limit of $\Im{m}(H\sigma)$ also exists and is finite for almost all $\zeta\in\mathbb T$. See \cite[page 113]{FM1}. Thus, it follows from \eqref{trick-construction-inner-function} that $|I(\zeta)|=1$ for almost all $\zeta\in\mathbb T$,  meaning that $I$ is an inner function. Of course, we have $I\not\equiv 1$. Now, we easily check that 
\[
\frac{1-|I(z)|^2}{|1-I(z)|^2}=\Re e(H\sigma(z))=\int_{\mathbb T}\frac{1-|z|^2}{|\xi-z|^2}\,d\sigma(\xi),
\]
which implies by unicity of the Clark measure that $\sigma_1=\sigma$. Therefore $I$ satisfies the assumptions of Theorem~\ref{sec5-th}.
\end{Remark}

\begin{Corollary}\label{Cor-rk-cyclique}
 Let $I$ be an inner function, $I\not\equiv 1$, and assume that $\sigma_1$ is a discrete measure. Let $b=(1+I)/2$. 
 Then  $k_\lambda^b$ is a cyclic vector for $S_b$ for every $\lambda\in\mathbb D$. 
\end{Corollary}

\begin{proof} 
Let us prove that $k_\lambda^b$ satisfies the assumptions $(a),(b)$ and $(c)$ of Theorem~\ref{sec5-th}. First, using that $b=(1+I)/2$, straightforward computations show that 
\[
\frac{1-\overline{b(\lambda)}b(z)}{1-\overline{\lambda}z}=(1-I(z))\frac{1}{4}\cdot\frac{(1-\overline{I(\lambda)})}{1-\overline{\lambda}z}+\frac{1}{2}\cdot\frac{1-\overline{I(\lambda)}I(z)}{1-\overline{\lambda}z}\cdot
\]
In other words, $k_\lambda^b$ can be written as $k_\lambda^b=(1-I)g_1+g_2$, with $g_1=\frac{1}{4}(1-\overline{I(\lambda)})k_\lambda$ and $g_2=\frac{1}{2}k_\lambda^I$. In particular, $g_1,g_2\in H^\infty$ and $k_\lambda^b$ satisfies the assumption $(a)$. 
\par\smallskip
Observe now that $\Re e(1-\overline{b(\lambda)}b(z))\geq 0$ and $\Re e(1-\overline{\lambda}z)\geq 0$ for every $z\in\D$, which implies that the functions $1-\overline{b(\lambda)}b(z)$ and $1-\overline{\lambda}z$ are outer. See \cite[page 67]{MR1864396}. So $k_\lambda^b$ is outer as the quotient of two outer functions. It remains to check that $k_{\lambda}^b$ satisfies assumption $(c)$. But 
\[
|k_\lambda^b(\zeta_n)|=\left|\frac{1-b(\lambda)}{1-\overline{\lambda}\zeta_n}\right|\geq \frac{|1-b(\lambda)|}{1+|\lambda|},
\]
and the property $(c)$ follows from the fact that
 \[
\sum_{n\geq 1}\frac{1}{|I'(\zeta_n)|}=\sum_{n\geq 1}\sigma_1(\{\zeta_n\})\le\sigma_1(\T)<+\infty.
\]
% Observe that by \eqref{def-clark-measure}, the fact that $I(0)=0$ implies that $\sigma_1$ is a probability measure. 

Thus $k_\lambda^b$ satisfies the assumptions $(a),(b),(c)$ of Theorem~\ref{sec5-th}, and $k_\lambda^b$ is cyclic for $S_b$. 

\end{proof}
It should be noted that in Corollary~\ref{Cor-rk-cyclique}, the reproducing kernels $f=k_\lambda^b$, $\lambda\in \mathbb D$, which are cyclic for $S_b$, are such that $1/f\in H^\infty$.
As we already observed in Lemma~\ref{lem:multi-inver},
certain invertibility conditions for $f$ make cyclicity easier. Using Theorem \ref{Thm:main-b-1+I}, we now construct a family of functions $f$ which are cyclic for $S_b$ but are such that $1/f\notin H^2$. 

\begin{Example}
 Let $I$ be a non-constant inner function, and assume that $\sigma_1$ is a discrete measure.  Let $b=(1+I)/2$ and $f=(1+I)k_{\lambda}^I$ for some $\lambda\in\D$. Then $f$ is cyclic for $S_b$ and $1/f\notin L^2$.
\end{Example}

\begin{proof}
  First observe that 
\[
f=(1-I)(-k_\lambda^I)+2k_\lambda^I,
\]
so that $f=(I-I)g_1+g_2$, with $g_2=-2g_1=2k_\lambda^I \in H^\infty\cap K_I$. In particular, $f$ satisfies condition $(a)$ of Theorem~\ref{Thm:main-b-1+I}. Moreover, the function $f$ is outer as  the product of two outer functions (use the same arguments as in the proof of Corollary~\ref{Cor-rk-cyclique}). Finally, since $|f(\zeta_n)|=2|k_\lambda^I(\zeta_n)|\geq |1-I(\lambda)|$,  $f$ satisfies condition $(c)$. Hence by Theorem~\ref{Thm:main-b-1+I}, $f$ is cyclic for $S_b$. 
\par\smallskip
Let us now check that $1/f\not \in L^2$. First observe that there exist two positive constants $C_1$ and $C_2$ such that
\[
C_1 \frac{1}{|1+I(\zeta)|}\le
\frac{1}{|f(\zeta)|}\le C_2 \frac{1}{|1+I(\zeta)|}\quad\mbox{ for a.e. }\zeta\in\T,
\]
because 
\[
\frac{1-|I(\lambda)|}{2}\leq |k_\lambda^I(\zeta)|\leq \frac{2}{1-|\lambda|}\cdot
\]
Now assume that $1/f$ belongs to $L^2$. Then  $1/(1+I)\in L^2$. But since $1+I$ is outer, we get that $1/(1+I)\in H^2$. See \cite[page 43]{MR1864396}. As
\[
\frac{1}{1+I}=\overline{\frac{I}{1+I}}\quad\mbox{ for a.e. }\zeta\in\T,
\]
we deduce that $1/(1+I)$ also belongs to $\overline{H^2}$. Thus $1/(1+I)$ is constant, which is a contradiction.
\end{proof}

In the context of Corollary~\ref{Cor-rk-cyclique}, it is easy to see that $(a,b)$ forms a corona pair, and we have seen that $k_\lambda^b$ is cyclic for $S_b$. In fact, this cyclicity result holds true under the (HCR) condition only.

\begin{Proposition}\label{chose}
Let $b$ be a non-extreme point in $\operatorname{ball}(H^{\infty})$, and let $a$ be its pythagorean mate. Assume that $(a,b)$ satisfies (HCR). Then the following. assertions hold:
\begin{enumerate}
\item $k_\lambda^b$ is cyclic for $S_b$ for every $\lambda\in\mathbb D$;
\item If $b$ is furthermore assumed to be outer, then $bk_\lambda$ is also cyclic for $S_b$ for every $\lambda\in\mathbb D$. In particular, $b$ is a cyclic vector for $S_b$.
\end{enumerate}
\end{Proposition}

\begin{proof}
(a) We have $p(S_b)k_\lambda^b=(1-\overline{b(\lambda)}b)p(S_b)k_\lambda$ for every polynomial $p\in\C[X]$. Since $(a,b)$ satisfies (HCR), we have $\HH(b)=\MM(\bar a)$, and $b$ is a multiplier of $\HH(b)$. See \cite[Theorems 28.7 and 28.3]{FM2}. In particular, the multiplication operator $T=M_{1-\overline{b(\lambda)}b}$ is bounded on $\HH(b)$ and we have
\[
p(S_b)k_\lambda^b=Tp(S_b)k_\lambda\qquad \textrm{for every polynomial }p\in\C[X].
\]
Since $k_\lambda$ is cyclic for $S_b$ (see Example~\ref{example-rkhardycyclic}), in order to check that $k_\lambda^b$ is also cyclic for $S_b$ it is sufficient to check that $T$ has dense range. Let $h\in\HH(b)$ be such that $h\perp \mbox{Range}(T)$. Then $h\perp Tk_\mu=k_\mu-\overline{b(\lambda)}bk_\mu$ for every $\mu\in\mathbb D$. Lemma~\ref{Lem:completenes} now implies that $h=0$, proving that $T$ has dense range. It follows that $k_\lambda^b$ is cyclic for $S_b$.

(b) The proof of (b) proceeds along the same lines of (a). We have
\[
p(S_b)bk_\lambda=Vp(S_b)k_\lambda\qquad \textrm{for every polynomial }p\in\C[X],
\]
where $V=M_b$ is the multiplication operator by $b$. As previously, in order to show that $bk_\lambda$ is cyclic for $S_b$, it is sufficient to check that $V$ has a dense range. Let $h\in\HH(b)$ be such that $h\perp \mbox{Range}(V)$. Then $h\perp Vk_\mu=bk_\mu$ for every $\mu\in\mathbb D$. By \eqref{eq1EZD:lem-completeness}, it then follows that $h^+(\mu)=0$ for every $\mu\in\mathbb D$. Then $h^+=0$ and $T_{\bar b}h=T_{\bar a}h^+=0$. But, since $b$ is outer, $T_{\bar b}$ is one-to-one, which implies that $h=0$. It then follows that $bk_\lambda$ is cyclic for $S_b$. 
\end{proof}

We finish the paper with the following question:

\begin{question}
Does Proposition \ref{chose} hold true without the assumption that $(a,b)$ satisfies (HCR)?
\end{question}

%\begin{Corollary}\label{cor:main-b-1+I}
%Let $I$ be a one-component inner function, $I(0)=0$ and assume that $\sigma_1$ is a discrete measure. Let $\{\zeta_n:n\geq 1\}=\{\zeta\in E_0(I): I(\zeta)=1\}$. Let $b=(1+I)/2$ and $f\in\HH(b)$. Assume that:
%\begin{enumerate}
%\item $f$ is outer;
%\item we have 
%\[
%\sum_{n\geq 1}\frac{1}{|f(\zeta_n)| |I'(\zeta_n)|^{1/2}}<\infty.
%\]
%\end{enumerate}
%Then $f$ is cyclic for $S_b$. 
%\end{Corollary}
%

\bibliographystyle{plain}

\bibliography{references}

\begin{thebibliography}{10}

\bibitem{MR1367960}
Peter Borwein and Tam\'{a}s Erd\'{e}lyi.
\newblock {\em Polynomials and polynomial inequalities}, volume 161 of {\em
  Graduate Texts in Mathematics}.
\newblock Springer-Verlag, New York, 1995.

\bibitem{B2}
Leon Brown.
\newblock Invertible elements in the {D}irichlet space.
\newblock {\em Canad. Math. Bull.}, 33(4):419--422, 1990.

\bibitem{BC1}
Leon Brown and William Cohn.
\newblock Some examples of cyclic vectors in the {D}irichlet space.
\newblock {\em Proc. Amer. Math. Soc.}, 95(1):42--46, 1985.

\bibitem{BS1}
Leon Brown and Allen~L. Shields.
\newblock Cyclic vectors in the {D}irichlet space.
\newblock {\em Trans. Amer. Math. Soc.}, 285(1):269--303, 1984.

\bibitem{MR301534}
Douglas~N. Clark.
\newblock One dimensional perturbations of restricted shifts.
\newblock {\em J. Analyse Math.}, 25:169--191, 1972.

\bibitem{MR3110499}
Constantin Costara and Thomas Ransford.
\newblock Which de {B}ranges-{R}ovnyak spaces are {D}irichlet spaces (and vice
  versa)?
\newblock {\em J. Funct. Anal.}, 265(12):3204--3218, 2013.

\bibitem{MR0215065}
Louis de~Branges and James Rovnyak.
\newblock {\em Square summable power series}.
\newblock Holt, Rinehart and Winston, New York-Toronto, Ont.-London, 1966.

\bibitem{MR203465}
Ronald~G. Douglas, Harold~S. Shapiro, and Allen~L. Shields.
\newblock On cyclic vectors of the backward shift.
\newblock {\em Bull. Amer. Math. Soc.}, 73:156--159, 1967.

\bibitem{Duren}
Peter~L. Duren.
\newblock {\em Theory of ${H}\sp{p}$ spaces}.
\newblock Academic Press, New York, 1970.

\bibitem{MR3475457}
Omar El-Fallah, Youssef Elmadani, and Karim Kellay.
\newblock Cyclicity and invariant subspaces in {D}irichlet spaces.
\newblock {\em J. Funct. Anal.}, 270(9):3262--3279, 2016.

\bibitem{EFKR1}
Omar El-Fallah, Karim Kellay, and Thomas Ransford.
\newblock Cyclicity in the {D}irichlet space.
\newblock {\em Ark. Mat.}, 44(1):61--86, 2006.

\bibitem{EFKR2}
Omar El-Fallah, Karim Kellay, and Thomas Ransford.
\newblock On the {B}rown-{S}hields conjecture for cyclicity in the {D}irichlet
  space.
\newblock {\em Adv. Math.}, 222(6):2196--2214, 2009.

\bibitem{MR3503356}
Emmanuel Fricain, Andreas Hartmann, and William~T. Ross.
\newblock Concrete examples of {$\mathscr{H}(b)$} spaces.
\newblock {\em Comput. Methods Funct. Theory}, 16(2):287--306, 2016.

\bibitem{MR3850543}
Emmanuel Fricain, Andreas Hartmann, and William~T. Ross.
\newblock Range spaces of co-analytic {T}oeplitz operators.
\newblock {\em Canad. J. Math.}, 70(6):1261--1283, 2018.

\bibitem{MR3967886}
Emmanuel Fricain, Andreas Hartmann, and William~T. Ross.
\newblock Multipliers between range spaces of co-analytic {T}oeplitz operators.
\newblock {\em Acta Sci. Math. (Szeged)}, 85(1-2):215--230, 2019.

\bibitem{MR2390675}
Emmanuel Fricain and Javad Mashreghi.
\newblock Boundary behavior of functions in the de {B}ranges-{R}ovnyak spaces.
\newblock {\em Complex Anal. Oper. Theory}, 2(1):87--97, 2008.

\bibitem{FM1}
Emmanuel Fricain and Javad Mashreghi.
\newblock {\em The theory of {$\mathcal{H}(b)$} spaces. {V}ol. 1}, volume~20 of
  {\em New Mathematical Monographs}.
\newblock Cambridge University Press, Cambridge, 2016.

\bibitem{FM2}
Emmanuel Fricain and Javad Mashreghi.
\newblock {\em The theory of {$\mathcal{H}(b)$} spaces. {V}ol. 2}, volume~21 of
  {\em New Mathematical Monographs}.
\newblock Cambridge University Press, Cambridge, 2016.

\bibitem{MR3309352}
Emmanuel Fricain, Javad Mashreghi, and Daniel Seco.
\newblock Cyclicity in non-extreme de {B}ranges--{R}ovnyak spaces.
\newblock In {\em Invariant subspaces of the shift operator}, volume 638 of
  {\em Contemp. Math.}, pages 131--136. Amer. Math. Soc., Providence, RI, 2015.

\bibitem{MR3526203}
Stephan~Ramon Garcia, Javad Mashreghi, and William~T. Ross.
\newblock {\em Introduction to model spaces and their operators}, volume 148 of
  {\em Cambridge Studies in Advanced Mathematics}.
\newblock Cambridge University Press, Cambridge, 2016.

\bibitem{garnett}
John~B. Garnett.
\newblock {\em Bounded analytic functions}, volume~96 of {\em Pure and Applied
  Mathematics}.
\newblock Academic Press, Inc. [Harcourt Brace Jovanovich, Publishers], New
  York-London, 1981.

\bibitem{HS1}
H{\aa}kan Hedenmalm and Allen Shields.
\newblock Invariant subspaces in {B}anach spaces of analytic functions.
\newblock {\em Michigan Math. J.}, 37(1):91--104, 1990.

\bibitem{MR3390195}
Karim Kellay and Mohamed Zarrabi.
\newblock Two-isometries and de {B}ranges--{R}ovnyak spaces.
\newblock {\em Complex Anal. Oper. Theory}, 9(6):1325--1335, 2015.

\bibitem{MR4246975}
Shuaibing Luo, Caixing Gu, and Stefan Richter.
\newblock Higher order local {D}irichlet integrals and de {B}ranges--{R}ovnyak
  spaces.
\newblock {\em Adv. Math.}, 385:Paper No. 107748, 47, 2021.

\bibitem{MR1864396}
Nikolai~K. Nikolski.
\newblock {\em Operators, functions, and systems: an easy reading. {V}ol. 1},
  volume~92 of {\em Mathematical Surveys and Monographs}.
\newblock American Mathematical Society, Providence, RI, 2002.
\newblock Hardy, Hankel, and Toeplitz, Translated from the French by Andreas
  Hartmann.

\bibitem{RS1}
Stefan Richter and Carl Sundberg.
\newblock Multipliers and invariant subspaces in the {D}irichlet space.
\newblock {\em J. Operator Theory}, 28(1):167--186, 1992.

\bibitem{MR847333}
Donald Sarason.
\newblock Doubly shift-invariant spaces in {$H^2$}.
\newblock {\em J. Operator Theory}, 16(1):75--97, 1986.

\bibitem{MR1289670}
Donald Sarason.
\newblock {\em Sub-{H}ardy {H}ilbert spaces in the unit disk}, volume~10 of
  {\em University of Arkansas Lecture Notes in the Mathematical Sciences}.
\newblock John Wiley \& Sons, Inc., New York, 1994.
\newblock A Wiley-Interscience Publication.

\end{thebibliography}

\end{document}